\documentclass{amsart}[draft,11pt]
\usepackage[utf8]{inputenc}
\usepackage{a4wide}
\usepackage{amsmath,amsfonts,amssymb,amsthm}
\usepackage{mathtools}

\usepackage{xcolor,verbatim}

\newcommand{\mg}{\mathfrak{g}}
\newcommand{\mh}{\mathfrak{h}}
\newcommand{\ms}{\mathfrak{s}}
\newcommand{\mk}{\mathfrak{k}}

\newcommand{\myrightleftarrows}[1]{\mathrel{\substack{\xrightarrow{#1} \\[-.9ex] \xleftarrow{#1}}}}

\newcommand{\weyl}{\mathcal{W}(\ms^\mathbb{C})}

\newcommand{\fg}{{\mathfrak{g}}}
\newcommand{\fh}{{\mathfrak{h}}}
\newcommand{\fk}{{\mathfrak{k}}}
\newcommand{\fo}{{\mathfrak{o}}}
\newcommand{\fSp}{{\mathfrak{so}}}

\newcommand{\fp}{{\mathfrak{p}}}
\newcommand{\fs}{{\mathfrak{s}}}

\newcommand{\CC}{{\mathbb{C}}}

\newcommand{\RR}{{\mathbb{R}}}

\newcommand{\bZ}{{\mathbb{Z}}}
\newcommand{\bR}{{\mathbb{R}}}

\newcommand{\cW}{{\mathcal{W}}}

\newcommand{\ad}{{\mathrm{ad}}}

\newcommand{\fe}{{\mathfrak{e}}}

\usepackage{marginnote}

\usepackage{color}
	\definecolor{bethorange}{rgb}{0.98, 0.3, 0.0}
	\definecolor{amblue}{rgb}{0.32,0.09,0.98}

\numberwithin{equation}{section}

\theoremstyle{definition}
\newtheorem{theorem}{Theorem}[section]
\newtheorem{thm}[theorem]{Theorem}

\newtheorem{definition}[theorem]{Definition}

\newtheorem{lemma}[theorem]{Lemma}
\newtheorem{lem}[theorem]{Lemma}

\newtheorem{corollary}[theorem]{Corollary}
\newtheorem{proposition}[theorem]{Proposition}
\newtheorem{remark}[theorem]{Remark}
\newtheorem{rem}[theorem]{Remark}

\newtheorem{example}[theorem]{Example}

\title[Symplectic Dirac operators on homogeneous spaces]{Symplectic Dirac operators on homogeneous spaces}

\author
{Dan Ciubotaru}
        \address[D. Ciubotaru]{Mathematical Institute, University of Oxford, Oxford OX2 6GG, UK}
        \email{dan.ciubotaru@maths.ox.ac.uk}

\author{Gang Liu}
\address[G. Liu]{Université de Lorraine, CNRS, IECL, F-57000 Metz, France}
\email{gang.liu@univ-lorraine.fr}

\author{Salah Mehdi}
\address[S. Mehdi]{Université de Lorraine, CNRS, IECL, F-57000 Metz, France}
\email{salah.mehdi@univ-lorraine.fr}

\author{Nicolas Prudhon}
\address[N. Prudhon]{Université de Lorraine, CNRS, IECL, F-57000 Metz, France}
\email{nicolas.prudhon@univ-lorraine.fr}

\begin{document}

\maketitle

\begin{abstract}
We define symplectic Dirac operators on homogeneous spaces and study their representation-theoretic role. For an invariant polarization, the symplectic Dirac operator decomposes into two symplectic Dolbeault operators. We compute their commutator as the natural symplectic analogue of the square of the classical Dirac operator. Our first main result gives a necessary and sufficient condition for this commutator to satisfy a Parthasarathy-type formula. We further prove that, whenever this condition fails, no cubic perturbation of the symplectic Dolbeault operators can yield such a formula, in contrast with Kostant’s cubic Dirac operator in the orthogonal setting. As applications, we establish an $\mathfrak{sl}_2$-structure generated by the symplectic Dolbeault operators and derive Dirac-type inequalities for unitary representations of Hermitian symmetric spaces labelled by the levels of the symmetric powers of the antiholomorphic tangent space at the identity. The level-zero inequality recovers the standard Parthasarathy–Dirac inequality, while the higher levels inequalities yield new constraints. For $SU(1,n)$, we show that, for representations with a specific Kraljevi\'c corner, the level-one inequality strengthens all basic Parthasarathy inequalities of the first kind for particular $K$-types, precisely those satisfying an explicit highest-weight condition.
\end{abstract}

\section*{Introduction}
The classical Dirac operator is a central object linking representation theory, differential geometry and analysis. Originally introduced in quantum mechanics, its mathematical importance became apparent via the Atiyah--Singer Index Theorem, which related the analytical properties of elliptic differential operators and the topology of manifolds. In representation theory, the Dirac operator provides a powerful tool for constructing basic representations of semisimple Lie groups. The pioneering work of Parthasarathy and Atiyah--Schmid realized the discrete series representations of a semisimple Lie group geometrically as certain spaces of harmonic spinors. Moreover, Parthasarathy introduced the fundamental Dirac inequality that has served as a useful technique for studying unitary representations.

\smallskip
Subsequent work expanded the role of Dirac theory in  modern representation theory. Kostant introduced a cubic Dirac operator for Lie algebras, giving a new perspective on Lie algebra cohomology.  Vogan formulated the notion of Dirac cohomology, and conjectured that the Dirac cohomology of an irreducible $(\mathfrak{g},K)$-module determines its infinitesimal character. This conjecture, proved by Huang and Pand\v zi\' c, has led to a resurgence in the application of Dirac-operator techniques to unitary representation theory. 

\smallskip
In parallel with the classical theory associated with Riemannian spin geometry, symplectic analogues of Dirac operators have gradually emerged. Symplectic spinors were introduced by Kostant \cite{kostant1974} through the oscillator representation of the metaplectic group, replacing the Clifford algebra by the Weyl algebra. Habermann \cite{habermann,HH,BHH09}  defined the symplectic Dirac
operator on symplectic manifolds admitting metaplectic structures, and the geometry and analysis of these operators were subsequently developed by Habermann and Habermann, Cahen, Gutt and al \cite{CG2011,CG2013}. 

\smallskip
While the construction of the symplectic Dirac operator resembles that of its orthogonal counterpart, its algebraic and analytical properties are fundamentally different. The underlying Weyl algebra is infinite dimensional and does not possess many of the structural features of  Clifford algebras.
For example, the square of the symplectic Dirac operator does not admit an obvious expression in terms of Casimir operators, so many of the arguments that underlie the classical theory no longer apply. 
Nevertheless, once an invariant polarization is chosen, the
symplectic Dirac operator decomposes into two first-order operators, which may be regarded as symplectic analogues of Dolbeault operators. This decomposition indicates that the fundamental object is not the square of the Dirac operator itself but rather the commutator of
its Dolbeault components. One can ask whether this commutator admits an analogue of Parthasarathy's formula and, if so, whether it gives rise to new representation-theoretic invariants and unitary criteria.

\smallskip
This is the starting point of the present paper. The symplectic Dirac operator on a symplectic homogeneous space $G/H$ with an invariant polarization decomposes naturally into two symplectic Dolbeault operators
\[
D=D^{+}+D^{-},
\]
whose commutator turns out to possess a rich algebraic structure. This point of view is reminiscent of K\"ahler geometry and the decomposition of the de Rham operator into Dolbeault operators. 
The principle is that the commutator
$[D^{+},D^{-}]$
should be regarded as the fundamental invariant of symplectic Dirac theory, playing the role of $D^{2}$ in the orthogonal setting.

\smallskip
The first main result of the paper, Theorem \ref{Dirac-Parthasarthy}, confirms this approach. We introduce a compatibility condition between a homogeneous decomposition $\mathfrak g=\mathfrak h\oplus\mathfrak s$ and an invariant polarization, which we call the \emph{$J$-symmetric condition}, and prove that it characterizes the existence of a Parthasarathy-type formula. More precisely,
the commutator of the symplectic Dolbeault operators admits an identity in terms of the Casimir operators of $\mathfrak g$ and $\mathfrak h$, together with the natural diagonal action of $\mathfrak h$, if and only if the pair is $J$-symmetric, see Proposition \ref{poor-partha}. Similar arguments appear in the work of E. Korman \cite{korman2013}. A different algebraic approach to symplectic Dirac operators was developed
in \cite{CMM2023}, where a pair of $D^{\pm}$-operators for $\mathbb{Z}/2\mathbb{Z}$-graded
quadratic Lie algebras and graded affine Hecke algebras was constructed, along with commutator analogues of Parthasarathy's formula
and Casimir inequalities.

\smallskip

The second main result, Theorem \ref{rigidity}, concerns the rigidity of this construction. In the orthogonal theory, Kostant's cubic Dirac operator shows that Parthasarathy's formula can be extended beyond symmetric pairs by adding suitable cubic terms. It is natural to ask whether an analogous phenomenon occurs in the symplectic setting. Perhaps surprisingly, we show that the answer is negative: whenever the $J$-symmetric condition fails, no cubic perturbation of the
symplectic Dolbeault operators can eliminate the additional terms appearing in their commutator. Hence, there is no symplectic analogue of Kostant's cubic Dirac operator.

\smallskip
In addition to these structural results, the paper also develops several algebraic directions of independent interest. We show that, in some cases, the symplectic Dolbeault operators naturally generate an $\mathfrak{sl}_{2}$-action on oscillator modules, leading to an explicit description of large families of harmonic symplectic spinors. This observation is useful for studying the kernels of the symplectic Dolbeault operators and hints at the relationship between symplectic Dirac operators, Howe duality, and classical invariant theory.

\smallskip

The symplectic Dirac operator has immediate applications to representation theory. We specialize to Hermitian symmetric pairs {$({\mathfrak g},{\mathfrak k})$}, where the $J$-symmetric condition is automatically satisfied.  The
commutator of the symplectic Dolbeault operators leads to a family of Dirac-type inequalities for unitary $({\mathfrak g},K)$-modules. Unlike the classical Dirac inequality, which involves only the spin representation, the symplectic inequalities also depend on the degree of the oscillator representation. The representation theory of symmetric powers, described by Schmid's decomposition, enters the picture through twisted Schmid modules. They replace the role played by
{the $K$-components of the }
spin modules in the orthogonal theory.

\smallskip

At level zero, the oscillator representation reduces to the trivial symmetric power, and our symplectic Dirac inequality coincides with the standard inequalities. Higher degrees reveal new phenomena that have no analogue in the Clifford setting. The corresponding inequalities involve the interaction between the $K$-types of the given representation and the decomposition of symmetric powers of the irreducible $K$-module defined by the antiholomorphic tangent space $\mathfrak{p}^{-}$ at the identity coset. This provides refined constraints on unitary representations and leads to explicit non-unitarity criteria for certain classes of modules, including lowest-weight modules. We illustrate these ideas for $SU(1,n)$. In particular, for unitary irreducible representations with a specific Kraljevi\'c corner, we give an explicit criterion, in terms of the relevant weights and nonvanishing components, characterizing exactly when the level-one symplectic Dirac inequality improves all basic Parthasarathy inequalities of the first kind. It is natural to expect that further investigation at higher levels will lead to sharper results.

\smallskip 
We hope that the methods developed here provide a foundation for a systematic representation-theoretic study of symplectic Dirac operators. In particular, the appearance of natural $\mathfrak{sl}_2$-actions, the structure of the kernels of the symplectic Dolbeault operators, and the resulting Dirac inequalities suggest the existence of a richer theory of symplectic Dirac (or Dolbeault) cohomology, paralleling the role played by Vogan’s Dirac cohomology in the orthogonal setting. 

\medskip

The paper is organized as follows. In Section 1, we review the necessary properties of the Weyl algebra and establish several algebraic identities that will be used throughout the paper. In Section 2, we introduce symplectic Dirac and symplectic Dolbeault operators on homogeneous spaces, prove the symplectic Parthasarathy formula, establish the $\mathfrak{sl}_2$-structure generated by the Dolbeault operators, and show that no cubic perturbation can restore the Parthasarathy formula outside the $J$-symmetric setting. Section 3 is devoted to Hermitian symmetric spaces, where we derive symplectic Dirac inequalities for unitary $(\mathfrak {g},K)$-modules, compare them with the standard Parthasarathy inequality, and investigate their consequences for lowest-weight representations. In Section 4, we study explicit examples for $SU(1,n)$, including the precise criterion under which the level-one inequality improves all basic Parthasarathy inequalities of the first kind.

\bigskip

\noindent{\bf Acknowledgements.} D.C. thanks Universit\'e de Lorraine for the support and hospitality during a visit in March 2022 when some of the initial research was conducted. G. L., S. M. and N. P. were partially supported by OpART - ANR-23-CE40-0016.

\section{The Weyl algebra}\label{s:Weyl}

Let us first review some basic facts about the Weyl algebra. Let $(V,\omega)$ be a real symplectic space. We extend $\omega$ linearly to the complexification $V^\mathbb{C}$ of $V$ and consider the Weyl algebra 
$\mathcal{W}(V^\mathbb{C},\omega)$,
defined as the unital complex algebra generated by $V$
with relations 
$$ uv-vu= 2i\omega(u,v)\,.$$
It is knwon that there is an isomorphism of vector spaces between the Weyl algebra $\mathcal{W}(V^\mathbb{C},\omega)$ and the symmetric algebra $S(V^\mathbb{C})$. Now we recall this isomorphism. Define, for $v\in V^\mathbb{C}$, a linear map
$\epsilon (v)$ and a derivation $\iota(v)$ on $S(V^\mathbb{C})$ by
$$ \epsilon (v) w = v\cdot w \quad (w\in S(V^\mathbb{C})\,)
\,,\quad \iota (v) 1=0\,,\quad \iota (v) w = i\omega (v,w)\,,\quad (w\in V^\mathbb{C})\,.$$
The map $\gamma (v) =  \epsilon (v) + \iota (v)$
satisfies
\begin{align*}
    [\gamma (v) , \gamma (w)]
    &= [\epsilon (v), \iota (w)] +
    [\iota (v), \epsilon (w)] \\
     &= 2i\omega(v,w)\,,
\end{align*}
so that $\gamma$ extends to an algebra
homomorphism 
$$ 
\gamma \colon 
\mathcal{W}(V^\mathbb{C},\omega) \to 
\mathop{End} 
(S(V^\mathbb{C}))\,.
$$
This also leads to an homomorphism 
$$ \eta \colon \mathcal{W}(V^\mathbb{C},\omega)
\to S(V^\mathbb{C})
$$
defined by $\eta(x)=\gamma(x)(1)$. This is an isomorphism of the underlying linear spaces (see \cite{kostant2001}). These two spaces will then 
be identified. This means that $\mathcal{W}(V^\mathbb{C},\omega)$ has two algebra structures. We will distinguish them by writing $xy$ for the Weyl product and $x\cdot y$ for the commutative symmetric product.

Let $A\in S^2(V)$. Then the commutator in the Weyl algebra defines a homomorphism $v\mapsto [A,v]$ of $V$, which actually sits in the symplectic Lie algebra 
$$ \mathfrak{sp}(V,\omega) = \{ T \in \mathrm{End}(V)\,,\, \omega(Tu,v)+\omega(u,Tv)=0\}\,.$$
In fact, the space $S^2(V)$ is generated by squares $x^2$ and 
\begin{align*}
\omega ([x^2,u],v)
= 
\omega(4i\omega(x,u)x,v)
=
4i\omega(x,u)\omega(x,v)
=
\omega(4i\omega(x,v)x,u)
=
-\omega(u,[x^2,v])\,.
\end{align*}

We then have a linear map
\begin{equation} 
\nu \colon S^2(V) \to \mathfrak{sp}(V)\,.
\label{def_nu}
\end{equation}
Moreover the Weyl commutator preserves $S^2(V)$ because 
$$ [x^2,y^2]= [x^2,y]y + y[x^2,y]
= 4i\omega(x,y)(xy+yx)\,,
$$
and this turns $S^2(V)$ into a Lie algebra.
The map $\nu$ is a Lie algebra homomorphism. Indeed, by the Jacobi identity, one has
$$ 
[\nu(A),\nu(B)](v)= 
[A,[B,v]]-[B,[A,v]] = [[A,B],v]=\nu([A,B])(v)\,.
$$
The map is in fact a Lie algebra isomorphism.

A \emph{(complex) polarization} of $V$
is given by $J\in \mathrm{End}_\mathbb{R}(V)$ 
satisfying the following conditions :
$$
    J^2=-\mathop{Id}\,,\qquad
    \omega(Jx,Jy)= \omega(x,y)\,.
$$
\begin{remark}
One can also introduce real polarizations, and write down parallel statements. A systematic treatment
of both real and complex polarizations, would have forced us to introduce more notations that are neither
comfortable nor useful in the applications we have in mind here.
\end{remark}
Letting
$$ B_J(x,y)=B_V(x,y) = \omega(x,Jy) .$$
$B_V$ is a non-degenerate symmetric form on $V^\mathbb{C}$. The endomorphism $J$ also defines  the Hermitian sesquilinear form $h\colon V\times V \to \mathbb{C}$ (equipped with the complex structure associated to $J$)  by
\begin{equation}
h(x,y) = B_V(x,y) -i\omega (x,y).
\label{def_h}
\end{equation}
\begin{remark}
There always exists a polarization $\tilde{J}$ such that $B_{\tilde{J}}$ is positive definite. In such a case, $\tilde{J}$ will be said \emph{positive}. In this case the hermitian form $h$ is positive definite as well.
\end{remark}

Let $J$ be any polarization. The Lie subalgebra $\mathfrak{u}(V,J)$
(or simply $\mathfrak{u}(V)$ if no confusion arises) of 
$\mathfrak{sp}(V,\omega)$ of operators commuting with
$J$ is the Lie algebra of skew-hermitian endomorphisms of $V$ with respect to $h$.

Let  $V^+$ (resp. $V^-$) be the eigenspace of $J$ in $V^\mathbb{C}$ with respect to the eigenvalue $i$ (resp.  $-i$). Then $V^+$ is Lagrangian for the symplectic form $\omega$, and  $V^-$ is its dual Lagrangian. Remark that  $V^+$ and  $V^-$  are $\mathfrak{u}(V)$-invariant. Let $(Z_k^+)$ be a basis of the complex Lagrangian subspace $V^+$ and 
$(Z_k^-)$ be the corresponding dual basis of $V^-$ with respect to $\omega$, that is 
$$
\omega(Z_k^+,Z_l^-)=\delta_{kl}\,.
$$
\begin{proposition}\label{unitary}
We have, for $w\in\nu^{-1}(\mathfrak{u}(V))$, for $\nu$ as in \eqref{def_nu},
$$
4w=i\sum_{kl}\omega([w,Z_k^+],Z_l^-)(Z_k^-Z_l^++Z_l^+Z_k^-)\,.
$$
\end{proposition}
\begin{proof}
The complex linear space $S^2(V^\mathbb{C})$ is generated by elements of the form :
$$ Z_i^+Z_j^+\,,\quad Z_i^-Z_j^-\,,\quad \text{and}\quad
Z_i^+Z_j^-+Z_j^-Z_i^+\,.$$
But $[Z_i^+Z_j^+,\cdot\,]$ and $[Z_i^-Z_j^-,\cdot\, ]$ map
$V^\pm$ to $V^\mp$ so they do not commute with $J$.
Hence $\nu^{-1}(\mathfrak{u}(V))$ is contained in the subspace generated by elements of the form  $(Z_k^-Z_l^++Z_l^+Z_k^-)_{kl}$ which are independent. Moreover
the dual basis with respect to 
$$\omega (x,y) = \sum_{\sigma \in \mathfrak{S}_N} \prod_{i=1}^{N} \omega (x_i,y_{\sigma(i)}) \quad \text{on}\quad S^N(V^\mathbb{C})\,,$$
is $(-\frac{1}{4}(Z_k^+Z_l^-+Z_l^-Z_k^+))_{kl}$ because
$$ \omega (Z_k^-Z_l^++Z_l^+Z_k^-,Z_r^-Z_s^++Z_s^+Z_r^-) = -4 \delta_{ks}\delta_{lr} \,.$$
We then have, for $w\in \nu^{-1}(\mathfrak{u}(V))$,
$$w=-\frac14\sum_{kl}\omega(w,Z_k^+Z_l^-+Z_l^-Z_k^+)
(Z_k^-Z_l^++Z_l ^+Z_k^-)\,.$$
But, if $w=x^2$,
\begin{align*}
\omega(x^2,Z_k^+Z_l^-+Z_l^-Z_k^+)
&=4\omega(x,Z_k^+)\omega(x,Z_l^-) \\
&=-i\omega(4i\omega(x,Z_k^+)x,Z_l^-)\\
&= -i\omega([x^2,Z_k^+],Z_l^-])\,,
\end{align*}
because $[x^2,Z_k^+]=[x,Z_k^+]x+x[x,Z_k^+]=4i\omega(x,Z_k^+)x$.
This ends the proof because elements
of the form $x^2$ generate $S^2(V^\mathbb{C})$. 
\end{proof}

Now we want to give a  very brief description of the Weyl algebra  $\mathcal{W}(V,\omega)$ via the two Lagrangians $V^\pm$, the procedure  called  the \emph{Fock picture} of Weyl algebra.

For  $v\in V^\mathbb{C}$, write $a(v)=-\frac12(v - iJv)$ and $c(v) = \frac12 ( v + iJv)$. Then we can verify directly that $a(v)\in V^+$ and $c(v)\in V^-$. 
On the other hand, it is clear that $v=c(v) - a(v)$. This decomposition defines in particular an involutive antilinear antiautomorphism on $\mathcal{W}(V^\mathbb{C},\omega)$ verifying $a(v)^* = c(v)$. Now we can give a brief presentation of the Fock picture :

\begin{enumerate}
\item The Weyl algebra is exactly  the involutive complex algebra generated by the elements $c(v)$ (for $v\in V$) with respect to the  relations (denoting $a(v)=c(v)^*$) :
$$
[c(v),c(w)]=[a(v),a(w)]=0\,,\quad [a(v),c(w)]=h(v,w)\,.
$$
where $h$ is as in \eqref{def_h}.
 \item Let $S(V^-)$ be the symmetric algebra generated by $V^-$. We have an injective morphism of involutive algebras
$\varphi \colon \mathcal{W}(V^\mathbb{C},\omega) \mathop{\longrightarrow}\limits \mathrm{End}(S(V^-))$
given by 
\begin{align*}
 \varphi(a(v))(c(w))&=[a(v),c(w)]=h(v,w)\,,
\\ 
\varphi(c(v))(c(w_1)\cdots c(w_n))&=c(v)c(w_1)\cdots c(w_n)\,.   
\end{align*}
and the fact that $\varphi(a(v))$ are derivations of $S(V^-)$.

\item  We have, for $x\in V$ and $w\in S^2(V)$
$$ a([w,x]) = [w,a(x)]\,,\qquad
c([w,x]) = [w,c(x)]\,.$$
\end{enumerate}

\begin{remark}
The representation $\varphi$ is unitary if, and only
if, the polarization $J$ is positive. 
Here it is to be understood that
the Hermitian form $h$ on $V$ is moved to $V^-$ thanks to
the antilinear isomorphism $v\mapsto c(v)$, 
and then on $S(V^-)$
by symmetrization. This obviously defines a 
pre Hilbert space structure on $S(V^-)$ if and only 
if $h$ is positive definite on $V$, 
that is, $J$ is a positive polarization.
In any case, the representation $\varphi$ is involutive in the sense that $\varphi(x^*)=\varphi(x)^*$. This is a true unitary representation of the Lie algebra $\mathfrak{sp}(V,\omega)$ as soon as $J$ is positive.
In other words, if $\tilde{J}$ is positive, and $J$
is any polarization, then the representation $\varphi$
associated to $\tilde{J}$ leads to a unitary representation of $\mathfrak{u}(V,J)$. If $J\neq\tilde{J}$, then the subspaces $S^N(\widetilde{V}^-)$ (with respect
to $\tilde{J}$) are not invariant subspaces of $\mathfrak{u}(V,J)$ and proposition \ref{unitary} fails however. 
\end{remark}

\section{A symplectic Dirac operator and a Parthasarathy formula}\label{s:partha}
\subsection{Symplectic Dirac operators}
Let $\mg$ be a real Lie algebra and  $\mh$ be a Lie sub-algebra 
of  $\mg$. Define  $\ms:=\mg/\mh $.  The action of $\mh$
on $\ms$ induced by the adjoint action of $\mh$ on $\mg$, leads to a morphism 
$$ 
\mathrm{ad} \colon 
\mh \to \mathfrak{gl}(\ms)\,.
$$
We assume that $\ms$ is endowed
with a non degenerate symplectic form $\omega$ (in particular, $\ms$ is even-dimensional), such that the adjoint action of
$\mh$ on it is symplectic:
$$
\forall x,y \in \ms\,,\quad w\in \mh\,,
\quad
\omega([w,x],y)+\omega(x,[w,y])=0\,,
$$
i.e., 
$$ 
\mathrm{ad} \colon 
\mh \to \mathfrak{sp}(\ms,\omega)\,.
$$

Let  $(x_k)$ be a basis of the complexification $\ms^\mathbb{C}$ of $\ms$ and $(x^k)$ be its dual basis with respect to $\omega$.
Assume moreover that: 
$$\mg = \mh \oplus \ms\text{ such that }[\mh,\ms]\subset \ms.$$
\begin{definition}
The symplectic Dirac operator associated with the data
$\fg$, $\fh$, $\ms$ and $\omega$ is given by 
\begin{equation*}\label{defdirac}
     D=\sum_k x^k\otimes x_k\in \mathcal{U}(\mg^\mathbb{C}) \otimes \weyl\,. 
\end{equation*}
\end{definition}
Following \eqref{def_nu}, we consider the isomorphism
$$ \nu \colon S^2(\ms) \to \mathfrak{sp}(\ms)\,.$$
Composing $\nu^{-1}$ with $\mathrm{ad}$ gives the morphism
$$\tau \colon \mh \mathop{\longrightarrow}\limits^\mathrm{ad} \mathfrak{sp}(\ms,\omega) \mathop{\longrightarrow}\limits^\mathrm{\nu^{-1}} \weyl
$$
as well as a morphism 
$$
\begin{array}{rccl}
\delta \colon &\mathcal{U}(\mh^\mathbb{C}) & \longrightarrow & \mathcal{U}(\mg^\mathbb{C}) \otimes \weyl \\ 
&X\in \mh^\mathbb{C} & \longmapsto & X\otimes 1 + 1\otimes \tau(X)\,.
\end{array}
$$

We have the following result.

\begin{proposition}
The symplectic Dirac operator $D$ is independent of the choice of the basis of $\ms^\mathbb{C}$.  Moreover, $D$ is $\mh$-invariant, i.e., $[\delta(\mh),D]=0$.
\end{proposition}
\begin{proof}
The independence of basis is clear. Next, let $X\in \mh$. We have 
\begin{align*}
    [\delta(X), D]
    &= [X\otimes 1 + 1\otimes \tau(X), \sum x^i \otimes x_i] \\
    &= \sum_i [X,x^i]\otimes x_i + \sum_i x^i \otimes [\tau(X),x_i] \\
    &= \sum_{ij} -\omega ([X,x^i],x_j)x^j \otimes x_i+ \sum_i x^i \otimes [\tau(X),x_i] \\
    &=\sum_{ij} \omega (x^i,[X,x_j])x^j \otimes x_i+ \sum_i x^i \otimes [\tau(X),x_i] \\
    &=\sum_{j}  x^j\otimes \sum_i \omega (x^i,[X,x_j])x_i + \sum_i x^i \otimes [\tau(X),x_i] \\
    &= \sum_i x^i \otimes 
    (-[X,x_i]+[\tau(X),x_i]) = 0\,.
\end{align*}
\end{proof}
Unlike the usual Clifford–Dirac operator, whose square differs from the Casimir operator by a constant according to Parthasarathy’s formula, the symplectic Dirac operator $D$ generally admits no such simple relation. However, a polarization of $\mathfrak{s}^{\mathbb C}$ yields a decomposition $D=D^++D^-$ for which the commutator $[D^+,D^-]$ satisfies an analogue of Parthasarathy’s formula. Establishing this identity is the aim of the next subsection.

\subsection{Invariant polarizations and symplectic Dolbeault operators}
\begin{definition}
\label{invpol}
A polarization $J$ of the symplectic space $\ms^{\mathbb C}$ is said to be \emph{invariant}
if it satisfies
$$ [\mathrm{ad}(w),J]=0 \,,\quad \text{for all }w\in \mh\,.$$

\end{definition}
Note that the two eigenspaces $\mathfrak{s}^\pm$ of $J$ on $\ms$ are $\mathfrak{h}$-invariant.
\begin{example}\label{parabolic}
On can take
$\mh^\mathbb{C}$ to be the Levi part of a parabolic subalgebra $\mathfrak{q}= \mh^\mathbb{C} \oplus
\ms^+$ of a reductive Lie algebra $\mathfrak{g}^\mathbb{C}=\mathfrak{h}^\mathbb{C}\oplus \ms^+\oplus \ms^-$.
Of particular interest is the case of a Hermitian symmetric pair
$(\fg,\fk)$.
\end{example}

Any polarization of $\ms$ provides a decomposition of the symplectic Dirac operator $D$.
Indeed, let $(Z_k^+,Z_k^-)$ be an adapted basis of $\ms$
as in the previous section. Then it is a basis of $\ms^\mathbb{C}$ whose dual basis
with respect to the symplectic form $\omega$
is $(Z_k^-,-Z_k^+)$.
Then one has
$$D=D^+ + D^- $$
where
$$
D^+= -\sum Z_k^+\otimes Z_k^-  
\text{ and } 
D^-= \sum Z_k^-\otimes Z_k^+ \,.
$$ 
The operators $D^\pm$ will be called \emph{symplectic Dolbeault operators}.  

\begin{remark}The operators $D^+$ and $D^-$ are independent of the choice of the basis $(Z_k^+)$ of $\ms^+$.
Moreover, similar arguments as above show that
$D^\pm$ are  $\delta(\mh)$-invariant as soon as $J$ is invariant.
\end{remark}

\begin{proposition}
\label{poor-partha}One has
$$
[D^+,D^-]= -\frac12\sum_{kl}  [Z_k^+,Z_l^-] \otimes (Z_k^-Z_l^+ + Z_l^+Z_k^-)+
\sum_k (Z_k^-Z_k^++Z_k^+Z_k^-)\otimes i \,.
$$
\end{proposition}
\begin{proof}
We have 
\begin{align*}
-[D^+,D^-]
    &=\left[\sum_k Z_k^+\otimes Z_k^-,\sum_l Z_l^-\otimes Z_l^+ \right] \\
    &=\sum_{k,l} Z_k^+Z_l^- \otimes Z_k^-Z_l^+ 
    - \sum_{k,l} Z_l^-Z_k^+ \otimes Z_l^+Z_k^- 
    \,.
\end{align*}
But, in the Weyl algebra, one has $[Z_l^+,Z_k^-]= 2i\omega(Z_l^+,Z_k^-)=2i\delta_{kl}$. Thus one gets
\begin{align*}
-[D^+,D^-]
    &= \sum_{k,l} [Z_k^+,Z_l^-]\otimes Z_k^-Z_l^+ - \sum_k Z_k^-Z_k^+\otimes 2i\,.
\end{align*}
Similarly, one has
\begin{align*}
-[D^-,D^+]
    &=\left[\sum_l Z_l^-\otimes Z_l^+,\sum_k Z_k^+\otimes Z_k^- \right] \\
    &=\sum_{k,l} Z_l^-Z_k^+ \otimes Z_l^+Z_k^- 
    - \sum_{k,l} Z_k^+Z_l^- \otimes Z_k^-Z_l^+ \\
    &=  \sum_{k,l} [Z_l^-,Z_k^+]\otimes Z_l^+Z_k^- + \sum_k Z_k^+Z_k^-\otimes 2i
    \,.
\end{align*}
and it follows that
\begin{align*}
-[D^+,D^-] 
    &= \frac12 \left(-[D^+,D^-]+[D^-,D^+ ]\right) \\
    &= \frac12 \left( \sum_{kl} [Z_k^+,Z_l^-] \otimes (Z_k^-Z_l^+ + Z_l^+Z_k^-) \right. \\
    &\phantom{= \frac12 \left(\sum_{kl} [Z_k^+,Z_l^-] \otimes\right.} \left.-\sum_k (Z_k^-Z_k^++Z_k^+Z_k^-)\otimes 2i \right)\,.
\end{align*}
\end{proof}

\begin{example}
\label{heisenberg}
An immediate but meaningful application of this formula concerns Heisenberg groups. Let $\fg$ be the $3$-dimensional Heisenberg Lie algebra, with relations
$$ [X,Y]=Z \,,\quad [Z,X]=[Z,Y]=0\,.$$
Here $\mh$ is generated by $Z$, and $\ms$ by $X$ and $Y$. The symplectic form is given
by $\omega(X,Y)=1$. Moreover,
$Z^+=\frac{1}{\sqrt{2}}(X- iY)$ and
$Z^-=\frac{-i}{\sqrt{2}}(X+ iY)$. Let 
$H=X^2+Y^2 \in \mathcal{U}(\fg^\mathbb{C})$. We have,
$[Z^+,Z^-]=[X,Y]=Z$, and
\begin{align*}
    Z^+Z^-+Z^-Z^+ 
        &=\frac{-i}{2}((X-iY)(X+iY)+(X+iY)(X-iY)) 
        =-i(X^2+Y^2) = -iH\,.
\end{align*}
Since the Weyl algebra is the quotient of
$\mathcal{U}(\fg^\mathbb{C})$ by the relation $Z=2i$, we may also consider
$H$ as an element of the Weyl algebra, and then our formula for the commutator $[D^+,D^-]$ simply reads as follows:
$$
[D^+,D^-]= \frac{i}{2}Z\otimes H + H\otimes 1\,.
$$
We consider the oscillator representation $(\varphi, L^2(\mathbb{R}))$ of $\fg$, defined by
$\varphi(X)f(s)=-2isf(s)$, $\varphi(Y)f=\partial_s f$ and $\varphi(Z)f=\varphi([X,Y])f=2if$. It defines a representation of the Weyl algebra as well.
The operator $\mathcal{H}=\varphi(H)$ is the Hermite operator
$\partial_s^2-4s^2$ and has eigenfunctions $(h_n)_{n\geq 0}$ with eigenvalues $-2(2n+1)$, a total system of mutually orthogonal functions defined by
$$h_0(s)=e^{-s^2}\text{ and }h_{k+1}=(\partial_s-2s)h_k\,.$$
In fact, using the relation $[\partial_s,s]=1$, one checks by
induction that $(\partial_s+2s)h_k=-4kh_{k-1}$, and then 
deduce $\mathcal{H}h_k=-2(2k+1)h_k$.

Let us consider the representation $\Phi=\varphi\otimes \varphi$ of $\mathcal{U}(\fg^\mathbb{C})\otimes \mathcal{W}(\ms^\mathbb{C},\omega)$ on $L^2(\mathbb{R})\otimes L^2(\mathbb{R})$.
The eigenvectors of
$$\Phi([D^+,D^-])=\mathcal{H}\otimes 1 - 1\otimes \mathcal{H}\,,$$
are the functions $h_m\otimes h_n$ with eigenvalues $2((2n+1)-(2m+1))=4(n-m)$.
The spectrum is then $4\mathbb{Z}$, with infinite-dimensional eigenspaces, and the kernel is generated by the functions $\{h_n\otimes h_n\,, n\geq 0\}$. 

On the other hand, the kernels of the operators $\Phi(D^\pm)$ are also simple to compute. We obtain after a short calculation  
$$
\Phi(D^+)(h_m\otimes h_n)=-2imh_{m-1}\otimes h_{n+1}  \text{ and } \Phi(D^-)(h_m\otimes h_n)=2inh_{m+1}\otimes h_{n-1}\,.
$$ 
It follows that the functions  $\{h_0\otimes h_n\,,n\geq 0\}$ generate
$\ker \Phi(D^+)$, while $\ker \Phi(D^-)$ is generated by the functions $\{h_m\otimes h_0\,,m\geq 0\}$. Moreover, $\ker \Phi(D^+) \cap \ker \Phi(D^-) \subset \ker \Phi([D^+,D^-])$ is one-dimensional and is generated by the function
$h_0\otimes h_0$.

We now turn to the operator $D$ itself. First, for $\alpha \in \mathbb{N}$, the spaces 
$$
V_\alpha = \oplus_{m+n=\alpha} \mathbb{C}(h_m\otimes h_n)
$$
are invariant under $\Phi(D)$ and
$$ \Phi(D)(h_m\otimes h_n)=-2imh_{m-1}\otimes h_{n+1}+2inh_{m+1}\otimes h_{n-1}\,.$$
We deduce that 
\begin{align*}
&\ker \Phi(D)|_{V_\alpha}=0 
\quad\text{if }\alpha\text{ is odd, and }\\
&\ker \Phi(D)|_{V_\alpha}
\text{ is one dimensional if }\alpha\text{ is even}.   
\end{align*} 
More precisely, in the even case, the kernel is generated by the function 
$$
\sum_{k=0}^{\alpha/2} 
\left(\begin{array}{c}\alpha/2 \\ k\end{array}\right)
h_{\alpha-2k}\otimes h_{2k}\,.
$$ 
\end{example}

\subsection{Polarizations and kernels}
\label{invpolker}

From now on, when no confusion arises, we will simply write
$D^\pm$ instead of $(\pi\otimes\varphi)(D^\pm)$ when $(X\otimes S(\ms^-),\pi\otimes\varphi)$ is a $\mathcal{U}(\fg^\mathbb{C})\otimes \mathcal W(\ms^\mathbb{C},\omega)$-module.
Note that $X\otimes S^0(\ms^-) \subset \ker D^-$.
In this subsection, we present a method for constructing higher level elements in $\ker D^-$. As in the previous section, let $(V,\omega)$ be a real symplectic vector space of dimension $2n$, and consider a complex polarization $J$ of $V^{\mathbb C}$. Define the symplectic Dolbeault operators 
$$D^\pm \in \mathcal{U}(\tilde{\fg}^\mathbb{C})\otimes \mathcal{W}(V^\mathbb{C},\omega)\,,$$
associated to the Heisenberg Lie algebra $\tilde{\fg}$ built out of $V$, i.e., $\tilde{\fg}= V \oplus \mathbb{R}Z$ as a vector space
and $[X,Y]=\omega(X,Y)Z$ for $X,Y \in V$(all other
brackets being $0$). The operators 
$D^\pm \in \mathcal{U}(\tilde{\mathfrak g}^\mathbb{C})\otimes \mathcal W(V^\mathbb C,\omega)$ factor through the Weyl algebra so that
$$ D^\pm \in \mathcal{W}(V^\mathbb{C},\omega) \otimes \mathcal{W}(V^\mathbb{C},\omega)\,. 
$$
\begin{proposition}
\begin{enumerate}
    \item The symplectic Dolbeault operators $D^\pm$
    generate a Lie subalgebra isomorphic to $\mathfrak{sl}(2,\mathbb C)$.
    \item The Lie subalgebra $\mathfrak{u}(V,J)$ embeds diagonally
    in $\mathfrak{sp}(V^\mathbb{C})\oplus \mathfrak{sp}(V^\mathbb{C})$ and commute with the $\mathfrak{sl}(2,\mathbb C)$ copy generated by $D^\pm$.
\end{enumerate}
\end{proposition}
\begin{proof}
The first item follows from proposition \ref{poor-partha}
and the rest follows easily. Let us note $H=[D^+,D^-]$.
One actually has from proposition \ref{poor-partha} and the fact that $[Z_k^+,Z_l^-]=\omega(Z_k^+,Z_l^-)Z=\delta_{kl}Z$ 
in the Heisenberg Lie algebra $\tilde{\mathfrak g}$.
\begin{align*}
    [H,D^+]
    &=
    \left[
        -\sum_{k}  i \otimes (Z_k^-Z_k^+ + Z_k^+Z_k^-)
            + \sum_k (Z_k^-Z_k^++Z_k^+Z_k^-)\otimes i, 
        -\sum_l  Z_l^+\otimes Z_l^-
    \right] \\
    &=  i\sum_{kl}Z_l^+ \otimes [Z_k^-Z_k^+ + Z_k^+Z_k^-,Z_l^-] 
        -i\sum_{kl} [Z_k^-Z_k^+ + Z_k^+Z_k^-,Z_l^+] \otimes Z_l^- \,.
\end{align*}
But, $[Z_k^-Z_k^+,Z_l^-]=[Z_k^-,Z_l^-]Z_k^++Z_k^-[Z_k^+,Z_l^-]=2i\delta_{kl}Z_k^-$ and mutatis mutandis relations for other brackets. It remains
$$ [H,D^+]=-8\sum Z_k^+\otimes Z_k^-=8D^+\,.$$
An analogous computation yields $[H,D^-]=-8D^-$. Finally,
 using again a similar computation as above (up to a sign), one obtains that
$[X\otimes 1 + 1\otimes X, D^\pm] = 0 $ for all $X\in \mathfrak{u}(V,J)$.
\end{proof}

It turns out that the commuting pair $(\mathfrak{u}(V,J),\mathfrak{sl}(2,\mathbb C))$ in 
$\mathfrak{sp}(4n,\mathbb C)$, sits in the dual complex reductive pair $(\mathfrak{gl}(n,\mathbb C),\mathfrak{gl}(2,\mathbb C))$ of type II.
According to the theory of dual pairs, the restriction of the oscillator representation of $\mathfrak{sp}(4n,\mathbb R)$ to $\mathfrak{u}(V,J)\simeq\mathfrak{u}(p,q)$ 
 splits into unitary highest weight modules, and  their multiplicity spaces have the structure of a finite dimensional  $\mathfrak{u}(2)$-module. See Kashiwara-Vergne \cite[Theorem 7.2]{KV}.

\

Let us give more details in the case where $J$ is positive, within classical Weyl invariant theory. As before, still working with the Fock model of the harmonic oscillator, denote by $S^m(V^-)$ the $m$-th symmetric power,  and by ${\bigwedge}^m V^-$ the $m$-th exterior power. For every $d\ge 1$, let $\mathfrak{S}_d$ be the symmetric group in $d$ letters. For every partition $\lambda$ of $m$, let $c_\lambda\in \mathbb C \mathfrak{S}_d$ be the Young symmetrizer \cite[\S4]{FH}. The tensor product $(V^-)^{\otimes d}$ has a diagonal action by $GL(V^-)$ which commutes with the action of $\mathfrak{S}_d$:
\[
(v_1\otimes\dots\otimes v_d)\cdot \sigma=v_{\sigma(1)}\otimes\dots\otimes v_{\sigma(d)},\quad \sigma\in \mathfrak{S}_d
\]\,.
By Schur-Weyl duality (\cite[\S6]{FH}), 
\[
\mathbb S_\lambda(V^-)=(V^-)^{\otimes d}\cdot c_\lambda
\]
is an irreducible $GL(V^-)$-representation. 
For example, $\mathbb S_{(d)}(V^-)=S^d(V^-)$ and $\mathbb S_{(1,\dots,1)}(V^-)={\bigwedge}^d V^-.$ By Pieri's rule, one has:
\begin{equation}
    S^m(V^-)\otimes S^k(V^-)=\bigoplus_{i=0}^{\min(m,k)} \mathbb S_{(m+k-i,i)}(V^-),\quad \text{ as }GL(V^-)\text{-representations}.
\end{equation}
Next, consider the contraction map ($k\ge 1$)
\[d_{m,k}:S^m(V^-)\otimes S^k(V^-)\to S^{m+1}(V^-)\otimes S^{k-1}(V^-),\quad d_{m,k}(f\otimes g)=\sum_{i} Z_i^- f\otimes Z_i^+ g,
\]
where $(Z_i^+)$ is still a basis of $V^+$ and $(Z_i^-)$ the dual basis of $V^-$. This definition does not depend on the choice of the basis and $d_{m,k}$ is $GL(V^-)$-equivariant.

One can see the definition of $d_{m,k}$ more canonically as the composition of the maps
\[
S^m(V^-)\otimes S^k(V^-)\xrightarrow{\text{id}_m\otimes \Delta\otimes\text{id}_k} S^m(V^-)\otimes V^-\otimes (V^-)^*\otimes S^k(V^-)\xrightarrow{c\otimes a}S^{m+1}(V^-)\otimes S^{k-1}(V^-),
\]
where 
$$\Delta \colon \mathbb C\to V^-\otimes V^+,\;\text{with }\Delta(1)=\sum_i Z_i^-\otimes Z_i^+$$ 
is the coevaluation map, and as before, $c$ is the creation map
$$
\begin{array}{rclcl}
     S^m(V^-)&\otimes& V^- & \longrightarrow  & S^{m+1}(V^-)\\
     c(w_1)\cdots c(w_m)&\otimes& c(v)& \longmapsto &
     c(w_1)\cdots c(w_m)\cdot c(v)
\end{array}
$$ 
and $a$ is the annihilation map
$$
\begin{array}{rclcl}
    V^+&\otimes& S^k(V^-)& \longrightarrow  & S^{k-1}(V^-)\\
    a(v) &\otimes & c(w_1)\cdots c(w_k)& \longmapsto & \sum_{i=1}^k h(v,w_i) c(w_1)\dotsb \widehat{c(w_i)}\dotsb c(w_k)\,.
\end{array} 
$$

\begin{lem} Suppose $\dim V^-\ge 2$. Then
\[\ker d_{m,k}=\begin{cases} \mathbb S_{(m,k)}(V^-),& m\ge k,\\ 0,&m<k.
\end{cases}
\]
\end{lem}

\begin{proof}
This is well known, see for example \cite[Exercise 6.20]{FH}, but we couldn't find a precise reference, so we sketch a proof. Notice that the domain and the codomain of $d_{m,k}$ decompose with multiplicity one and the only summand that appears in the domain, but not in the codomain is $\mathbb S_{(m,k)}(V^-)$, when $m\ge k$. Suppose $m\ge k.$ Since $\partial=d_{m,k}$ is $GL(V^-)$-equivariant, it is immediate that $\mathbb S_{(m,k)}(V^-)\subseteq \ker \partial$, and for all $0\le i<k$, $\partial$ acts by a multiple of the identity on $\mathbb S_{(m+k-i,i)}(V^-)$. It remains to see that $\partial$ is nonzero on each $\mathbb S_{(m+k-i,i)}(V^-)$, $0\le i<k$.

 Let 
 $$v_i=\sum_{j=0}^i (-1)^j {\binom{i}{j}} (Z_1^-)^{m-i+j}(Z_2^-)^{i-j}\otimes (Z_1^-)^{k-j}(Z_2^-)^j$$ be a highest weight vector in $\mathbb{S}_{(m+k-i,i)}(V^-)$. Then it is clear that $\partial(v_i)\neq 0$, for example the summand 
$$
\partial((Z_1^-)^m\otimes (Z_1^-)^{k-i}(Z_2^-)^i)=(k-i) (Z_1^-)^{m+1}\otimes (Z_1^-)^{k-i-1}(Z_2^-)^i+i(Z_1^-)^m (Z_2^-)\otimes (Z_1^-)^{k-i} (Z_2^-)^{i-1}\,,
$$
contains the only term with $(Z_1^-)^{m+1}$ in the left hand side of the tensor. In fact, this calculation implies that $\partial$ acts as multiplication by $(k-i)$ on $\mathbb S_{(m+k-i,i)}(V^-)$, $0\leq i<k$.
\end{proof}
\begin{example}
\label{d11}
If we take $m=k=1$, then $\ker d_{1,1}={\bigwedge} ^2 V^-$. Explicitly, the kernel is spanned by elements 
\[\sum_{i,j} \varpi(Z_i^-,Z_j^-) \,Z_i^-\otimes Z_j^-,
\]
where $\varpi$ ranges over the skew-symmetric forms on $V^-$. 
\end{example}
\begin{remark}
We now have the desired decomposition.
Note that the image of $D^-$ by the representation $\Phi=\varphi\otimes\varphi$ is the endomorphism 
$$\partial=\oplus_{m,k}d_{m,k}\in \text{End}(S(V^-)\otimes S(V^-))\,,$$
and
the image of $D^+$ is $\partial^*$.
Let us write $V_{\lambda}$ the irreducible representation of $\mathfrak{u}(n)$ with highest weight $\lambda=(\lambda_1,\ldots,\lambda_n)$.

Let $n\geq 2$. The harmonic oscillator $\Phi$ of $\mathfrak{sp}(4n,\mathbb{R})$ 
decomposes under the dual pair $(\mathfrak u(n),\mathfrak{u}(2))$ as
$$ \Phi = \bigoplus_{N \ge 2i} V_{(N-i,i,0,\ldots,0)}\otimes  V_{(N-i,i)}\,,$$
Moreover, if we choose a highest weight vector $v^\lambda$ in $V_\lambda$, as well as a lowest weight vector $v_\lambda$, then the kernels of the symplectic Dolbeault $D^\pm$ 
are
$$
\ker D^- = \bigoplus_\lambda V_\lambda\otimes \mathbb C\cdot v_\lambda \,,\qquad
\ker D^+ = \bigoplus_\lambda V_\lambda\otimes \mathbb C\cdot v^\lambda\,.
$$
where the sums run over $\lambda$'s of the form
$\lambda=(\lambda_1,\lambda_2,0,\ldots,0)$.
Moreover, 
$$\ker D^- \cap \ker D^+=\oplus_{k \ge 0} V_{(k,k,0,\ldots,0)} \otimes {\det}^k\,.$$

This also holds when $n=1$.
The decomposition is however slightly different in this case because
the Young diagrams corresponding to the involved partitions have only one row, and the $\mathfrak{u}(2)$-modules appearing have highest weights
$(k,0)$. We then recover the results of example~\ref{heisenberg}.
\end{remark} 

It turns out that this discussion on the Heisenberg Lie algebra $\tilde{g}$ may be applied to general symplectic Dirac operators associated to $\mathfrak g$, $\mathfrak h$, $\omega$, ${\mathfrak s}$ and $J$, by taking $V^-=\mathfrak s^-$. Let us assume that $J$ is positive for simplicity and that
$\mathfrak s^-$ is an abelian Lie subalgebra 
of $\mathfrak g^{\mathbb C}$ so that 
$S(\mathfrak{s}^-) \subset \mathcal U(\mathfrak g^{\mathbb C})$ is an abelian subalgebra. 

\begin{proposition}\label{dolbeault-kernel}
Let $(\pi,V)$ be a $\mathfrak g$-module and $0\neq v_0\in V$. Let $N\ge 0$. For every $m\ge N$ and $f\in \mathbb S_{(m,N)}(\mathfrak s^-)\subset S^m(\mathfrak s^-)\otimes S^N(\mathfrak s^-)$, the element
\[
\label{vf}
v_f=(\pi\otimes \varphi)(f)\cdot (v_0\otimes 1)
\]
belongs to $\ker D^-\cap (V\otimes S^N(\mathfrak s^-))$. 
\end{proposition}

\begin{proof}
This follows by applying the previous lemma to $V^-=\mathfrak s^-$, $k=N$, and noting that $D^-(v_f)=(\pi\otimes \varphi)(d_{n,N}(f))\cdot (v_0\otimes 1)=0.$
\end{proof}

Notice that this construction does not guarantee that $v_f\neq 0$. There are however cases when this can be achieved at level one as we shall see in Section \ref{level1}. 

\subsection{Symplectic Parthasarathy formula}
Let us consider a polarization $J$ of
the symplectic space $\ms$ and the symmetric bilinear form $B_\ms$ associated to $J$,
$
B_\ms(x,y)=\omega(x,Jy)\,.
$
From now on we will assume that
$J$ is pseudo-Hermitian in the following sense.
\begin{definition}
A polarization $J$ of $\ms$ is said to be \emph{pseudo-Hermitian} if there exists a quadratic form $B_\mg$ on $\mg^\mathbb{C}$ such that
\begin{itemize}
    \item[(i)] $B_\mg$ is non degenerate and invariant,
    \item[(ii)] $B_\mg$ coincides with $B_\ms$ on $\ms^\mathbb{C}$,
    \item[(iii)] $\mh^\mathbb{C}$ and $\ms^\mathbb{C}$ are orthogonal with respect to $B_\mg$.
\end{itemize} 
\end{definition}
\noindent Note that a pseudo-Hermitian polarization $J$ is automatically invariant. 
Indeed, for $w\in \mh$ and $x,y\in \ms$, one has
$$
\omega(x,J[w,y])=B_\fg(x,[w,y])=-B_\fg([w,x],y)=-\omega(x,[w,Jy])\,.
$$
We will write  $B_\mh$ for the restriction of $B_\mg$ to $\mh^\mathbb{C}$. The bilinear form $B_\mh$ is non-degenerate, $\mathrm{ad}(\mh)$-invariant and symmetric on $\mh^\mathbb{C}$.  
\begin{example}
If the map
$$ \wedge^2 \ms \to \mh \,,\quad x\wedge y  \mapsto [x,y]_\mh$$
is onto, then $B_\fh$ is prescribed by the invariance condition:
$$B_\mh(w,[x,y])=B_\ms([w,x],y)\,.$$
Here $[x,y]_\mh$ is the projection of $[x,y]$ on $\mh$ with respect to the decomposition $\mg=\mh\oplus \ms$.
This condition is often fulfilled when $\mg$ is reductive. Note that the $\mathrm{ad}(\mh)$-invariance is then automatic thanks to the Jacobi identity. 

However, the full $\mathrm{ad}(\mg)$-invariance is not a consequence of this.
For example, if $\mg$ is the Heisenberg algebra and $\mh$ is the center of $\mg$, there does not exist such a form $B_\mathfrak{g}$.
\end{example}
By the orthogonality of $\mh$ and $\ms$ and the $\mathrm{ad}(\mg)$-invariance of $B_\mathfrak{g}$, it is clear that for $w\in \mh$ and $x,y\in \ms$, we have
$$
B_\mg(w,[x,y])=B_\ms([w,x],y)=\omega([w,x],Jy).
$$
Proposition \eqref{unitary} implies that, for $w\in \mathfrak{h}$:
\begin{equation}
    \label{tau}
4\tau(w)=-\sum_{kl}B_\mh(w,[Z_k^+,Z_l^-]_\mh)(Z_k^-Z_l^++Z_l^+Z_k^-).
\end{equation}

We now give a necessary and sufficient condition for the commutator $[D^+,D^-]$ to satisfy a Parthasarathy-type formula. Let $(w_j)$ be an orthonormal 
basis of $\mh^\mathbb{C}$ with respect to $B_\mh$, then  
$$\Omega_{\mh}= \sum_j w_j^2$$  
is the  Casimir operator of the quadratic Lie algebra $\mh^\mathbb{C}$. Moreover, since $B_\ms(Z_k^+,Z_l^-)=-i\delta_{kl}$ and $B_\ms(Z_k^\pm,Z_l^\pm)= 0 $, then
$$
\Omega_\mg = \Omega_{\mh} + i\sum_k (Z_k^+Z_k^- + Z_k^-Z_k^+)
$$
is the Casimir operator of $\mg^\mathbb{C}$.

\begin{definition}
A pair $(\mg,\mh)$, where $\mg=\mh \oplus \ms$ and
$\ms^{\mathbb C}=\ms^+\oplus \ms^-$ with respect to a pseudo-Hermitian polarization $J$, is called
\emph{$J$-symmetric}
if $[\ms^+,\ms^-]\subset \mh^\mathbb{C}$.
\end{definition}

\begin{theorem}
\label{Dirac-Parthasarthy}
The following assertions are equivalent.
\begin{enumerate}
\item The decomposition $\mg= \mh \oplus \ms$ 
is $J$-symmetric.
\item The symplectic Dirac operator $D$ satisfies 
the following formula
which we call the \emph{symplectic Parthasarathy formula} :
\begin{equation}
[D^+,D^-]= \Omega_{\mg}\otimes 1 + \delta(\Omega_{\mh})-1\otimes \tau(\Omega_{\mh})- 2\Omega_\mh \otimes 1 \,.
\label{partha}
\end{equation}
\end{enumerate}

\end{theorem}

\begin{proof}
By Proposition \ref{poor-partha}, we already have
$$
-[D^+,D^-]= \frac12\sum_{kl}  [Z_k^+,Z_l^-] \otimes (Z_k^-Z_l^+ + Z_l^+Z_k^-)
-    \Omega_{\mg}\otimes 1 +    \Omega_{\mh}\otimes 1\,.
$$
Writing
\begin{align*}
[Z_k^+,Z_l^-]&=
\sum_j B_\mh([Z_k^+,Z_l^-],w_j)w_j \\
&\phantom{\sum_j B_\mh([Z_k^+}+i\sum_m B_\ms([Z_k^+,Z_l^-],Z_m^-)Z_m^+ +
B_\ms([Z_k^+,Z_l^-],Z_m^+)Z_m^- \,,
\end{align*}
we finally have
\begin{equation} 
\label{partha-sum}
\begin{split}
[D^+,D^-]= &\,\Omega_{\mg}
\otimes 1 - \Omega_\mh \otimes 1 \\ &
-\sum_j w_j \otimes 
\left( \frac12\sum_{kl}B_\mh([Z_k^+,Z_l^-],w_j)
(Z_k^-Z_l^++Z_l^+Z_k^-)\right) \\ &
-\sum_{m}
Z_m^+\otimes \left(\frac{i}2 \sum_{kl} 
B_\ms([Z_k^+,Z_l^-],Z_m^-)
(Z_k^-Z_l^+ + Z_l^+Z_k^-) \right) \\ &
-\sum_l 
Z_m^-\otimes \left( \frac{i}2\sum_{kl}
B_\ms([Z_k^+,Z_l^-],Z_m^+)
(Z_k^-Z_l^+ + Z_l^+Z_k^-) \right)\,.
\end{split}
\end{equation}
Using \eqref{tau}, we also get that
\begin{equation} 
\label{partha-sum2}
\begin{split}
[D^+,D^-]= & \, \Omega_{\mg}\otimes 1 - \Omega_{\mh}\otimes 1 
+2\sum_j w_j \otimes \tau(w_j) \\ &
-\sum_{m}
Z_m^+\otimes \left(\frac{i}2 \sum_{kl} 
B_\ms([Z_k^+,Z_l^-],Z_m^-)
(Z_k^-Z_l^+ + Z_l^+Z_k^-) \right) \\ &
-\sum_l 
Z_m^-\otimes \left( \frac{i}2\sum_{kl}
B_\ms([Z_k^+,Z_l^-],Z_m^+)
(Z_k^-Z_l^+ + Z_l^+Z_k^-) \right)\,.
\end{split}
\end{equation}
On the one hand, we have:
$$
\delta(\Omega_{\mh}) = 
\Omega_{\mh} \otimes 1 + 2\sum w_j \otimes \tau(w_j) + 1\otimes \tau(\Omega_{\mh})\,.
$$ 
On the other hand, we observe directly that the last two sums vanish if, and only if,
 $[\ms^+,\ms^-]\subset \mh^\mathbb{C}$, i.e., exactly when the decomposition $\mg= \mh \oplus \ms$ is $J$-symmetric.
\end{proof}

\begin{rem}
\label{defd0}
Letting
$$ \mathcal{D}^0 
= 2 \sum_j w_j \otimes \tau(w_j)
=\delta(\Omega_\mathfrak{h})-\Omega_\mathfrak{h}\otimes 1-1\otimes \tau(\Omega_\mathfrak{h})
\, ,$$
Parthasarathy's formula (\ref{partha}) becomes
$$ [D^+,D^-]=\Omega_\mathfrak{g}\otimes 1-\Omega_\mathfrak{h}\otimes 1+\mathcal{D}^0\,.$$
\end{rem}

When the condition $[\ms^+,\ms^-] \subset \mh^\mathbb{C}$ is not satisfied, the commutator $[D^+,D^-]$ does not satisfy Parthasarathy's formula (\ref{partha}).
However, in this case, as in \cite{kostant2000}, it is natural to ask whether  there exist cubic terms $v^\pm \in \weyl$ such that the {\it shifted} operators 
$$\tilde{D}^\pm = D^\pm + v^\pm$$
satisfy a formula which is analogous to Kostant-Parthasarathy formula. More precisely,
$$
[\tilde{D}^+,\tilde{D}^-] = 
\Omega_{\mg}\otimes 1 + \delta(\Omega_{\mh})-1\otimes \tau(\Omega_{\mh})- 2\Omega_\mh \otimes 1+ 1\otimes [v^+,v^-]\,.
$$
The following statement is to be contrasted with Kostant's cubic Dirac operator.
\begin{theorem}\label{rigidity}
If the decomposition $\mg= \mh \oplus \ms$ is not $J$-symmetric, then there are no cubic terms $v^\pm \in \weyl$ such that
the operators 
$$\tilde{D}^\pm = D^\pm + v^\pm$$ satisfy the symplectic Kostant-Parthasarathy formula, i.e., for any $v^\pm \in \mathcal{W}(\mathfrak{s}^\mathbb{C})$:
$$
[\tilde{D}^+,\tilde{D}^-] \neq 
\Omega_{\mg}\otimes 1 + \delta(\Omega_{\mh})-1\otimes \tau(\Omega_{\mh})- 2\Omega_\mh \otimes 1+ 1\otimes [v^+,v^-]\,.
$$
\end{theorem}

\begin{proof}

Assume that the decomposition $\mg= \mh \oplus \ms$ is not $J$-symmetric and assume that there exist   $v^\pm \in \weyl$ such that
$$
[\tilde{D}^+,\tilde{D}^-] = 
\Omega_{\mg}\otimes 1 + \delta(\Omega_{\mh})-1\otimes \tau(\Omega_{\mh})- 2\Omega_\mh \otimes 1+ 1\otimes [v^+,v^-]
$$ 
where $\tilde{D}^\pm = D^\pm + v^\pm$.
According to the  equation \eqref{partha-sum} in Theorem \ref{Dirac-Parthasarthy}, we have
\begin{equation*}
[D^\pm,1\otimes v^\mp]=\mp
\frac{i}2\sum_{m}
Z_m^\pm \otimes \left( \sum_{kl} 
B_\ms([Z_k^+,Z_l^-],Z_m^\mp)
(Z_k^-Z_l^+ + Z_l^+Z_k^-) \right)\,.
\end{equation*}
So for each $m$, the cubic terms $v^\pm$  satisfy
the equation:
\begin{equation} 
\label{cubic-equation}  
[Z_m^\pm,v^\pm] =
i\sum_{kl} 
B_\ms([Z_k^+,Z_l^-],Z_m^\pm)
(Z_k^-Z_l^+ + Z_l^+Z_k^-),
\end{equation}
where the commutator on the left hand side stands for the one in $\weyl$.\\ Remember that $\gamma$ defines a linear map $\ms^\mathbb{C} \to \text{End}  (S(\ms^\mathbb{C}))$ that leads to an  
isomorphism 
$$\eta : \weyl \to S(\ms^\mathbb{C})$$ 
by setting $\eta(v)=\gamma(v)(1)$ that we use to identify $\weyl$ with  $S(\ms^\mathbb{C})$.
It should be understood that under this identification, there are two multiplications on $\weyl$.  For $u, v  \in  \weyl $, 
the Weyl multiplication will be denoted by $uv$ as we did previously, and the commutative multiplication of $S(\ms^\mathbb{C})$  will be denoted by $u\cdot v$.

Then we can check directly that 
\begin{equation} 
\label{weyl-sym}  
Z_k^\pm Z_l^\pm=Z_k^\pm \cdot Z_l^\pm, \ \text{and}  \   Z_k^\pm Z_l^\mp=Z_k^\pm \cdot Z_l^\mp \pm \delta_{kl}.
\end{equation} 
Hence for each $m$, the equation  \eqref{cubic-equation} becomes 
\begin{equation} 
\label{cubic-equation-sym}  
[Z_m^\pm,v^\pm] =
i\sum_{kl} 
B_\ms([Z_k^+,Z_l^-],Z_m^\pm)
(Z_k^-\cdot Z_l^+ + Z_l^+\cdot Z_k^-),
\end{equation} 
where the bracket on the left hand side stands for the commutator in $\weyl$.

Since $S(\ms^\mathbb{C})$  is isomorphic to algebra of polynomials and $Z_m^\pm$ is a basis of 
$\ms^\mathbb{C}$,  based on the relation \eqref{weyl-sym}, 
we deduce from the above equation \eqref{cubic-equation-sym} (for all $m$), $v^+$  necessarily is of  the form $v^+=v_1^+ +v_2^+$, 
where  
$$v_1^+ = \sum_{klm,m\leq k}a_{mkl} Z_m^-\cdot Z_k^-\cdot Z_l^+$$  and 
$$v_2^+ = \sum_{klm,m\leq k}b_{mkl} Z_m^+ \cdot Z_k^+\cdot Z_l^+.$$
Moreover, we can check directly that $[Z_m^+,v^+] =[Z_m^+, v_1^+]$. Thus, without loss of generality, we may assume 
$$v^+=v_1^+ = \sum_{klm,m\leq k}a_{mkl} Z_m^-\cdot Z_k^-\cdot Z_l^+.$$
We now have:
\begin{align*}
    [Z_n^+,v^+]
    &= \sum_{m\leq k,l}a_{mkl} [Z_n^+,Z_m^-\cdot Z_k^-\cdot Z_l^+] \\
    &= \sum_{m\leq k,l}a_{mkl} \left(
        [Z_n^+,Z_m^-]\cdot Z_k^-\cdot Z_l^+
        +Z_m^-\cdot [Z_n^+,Z_k^-]\cdot Z_l^+ \right)\\
    &= \sum_{k\geq n,l} 2ia_{nkl}Z_k^-\cdot Z_l^+ 
        + \sum_{m\leq n, l} 2ia_{mnl} Z_m^-\cdot Z_l^+ \\
    &= \sum_{k\geq n,l} 2ia_{nkl}Z_k^-\cdot Z_l^+      +\sum_{k\leq n,l} 2ia_{knl}Z_k^-\cdot Z_l^+\,.
\end{align*}
We observe that each of the $a_{knl}$ terms appears twice: 
once in $[Z_n^+,v^+]$ and another time 
in $[Z_k^+,v^+]$. We then must have
\begin{align*}
     2a_{knl} &= B_\ms([Z_k^+,Z_l^-],Z_n^+)
    =-B_\ms(Z_l^-,[Z_k^+,Z_n^+]) \\
    2a_{knl} &=
    B_\ms([Z_n^+,Z_l^-],Z_k^+)
    =-B_\ms(Z_l^-,[Z_n^+,Z_k^+]),
\end{align*}
which implies that $a_{knl}=-a_{knl}$. So $a_{knl}=0$ for all $a_{knl}$. Hence $v^+=0$. A similar treatment for $v^-$ also shows  that $v^-=0$. This is a contradiction.
\end{proof}

\begin{remark}
The operator $\tau(\Omega_{\mh})$
does not act by a constant. In fact it is a constant if and only if
$\mg=\mh\oplus \ms$ is a super Lie algebra \cite{kostant2001}. 
It does however act as a constant on $\tau(\mh)$ isotypic subspaces of $S(\ms^-)$.
\end{remark}

\section{Application to hermitian symmetric spaces}
\label{symmetrichermitian}

\subsection{The example of $\mathfrak{su}(1,1)$}
Before trying to handle the general case, let us briefly discuss the simplest case of $SL_2(\mathbb{R})$.

Let $(H,X,Y)$ be the standard $\mathfrak{sl}_2$ triple :
$$ [H,X]=2X\,,\quad [H,Y]=-2Y\,,\text{ and } [X,Y]=H\,.$$
The principal series of $SL(2,\mathbb{R})$ are 
parametrized by $(\epsilon,\nu)$, with $\epsilon=\pm 1$ and $\nu\in\mathbb{C}$.
The underlying Harish-Chandra module has a basis $w_k$, with $k$ even if $\epsilon=1$ and
$k$ odd if $\epsilon=-1$, and such that, 
\begin{equation*}
    \pi_{\epsilon,\nu}(H)w_k=kw_k\,,\quad
    \pi_{\epsilon,\nu}(X)w_k=\frac12(\nu+k+1)w_{k+2}\,,\quad
    \pi_{\epsilon,\nu}(Y)w_k=\frac12(\nu-k+1)w_{k-2}\,.
\end{equation*}
Let's take $\omega(X,Y)=1$. Then $S(\mathfrak{p}^-)=\mathbb{C}[Y]$ where
$\varphi(Y)$ acts by multiplication by $Y$
and $\varphi(X)$ acts as the derivation $\partial_Y$.
One gets
$$
[D^+,D^-](w_k\otimes Y^N)=
\left( -kN+\frac14(\nu^2-(k+1)^2)\right)(w_k\otimes Y^N)\,.
$$
Note that the operator $[D^+,D^-]$ is negative on the space $\ker (\pi_{\epsilon,\nu}\otimes \varphi)D^-$ as soon as $\pi_{\epsilon,\nu}$ is unitarizable.
We deduce the following facts :
\begin{itemize}
    \item 
If $\nu$ is not an integer, then 
$$\ker D^-=\sum_{k}\mathbb{C}w_k\otimes 1$$
is concentrated at level $N=0$.
We then have
$$ \nu^2 \leq (k+1)^2 \,\,(\text{for all  involved }k)\,.$$
In particular, if $\pi_{\epsilon,\nu}$ is unitarizable, then $\nu \in i\mathbb{R}\cup ]-1,1[$ if $\epsilon=1$, and $\nu^2<0$ if $\epsilon = -1$.
Moreover, $D^+$ is injective and 
$$\ker D^- \cap \ker D^+ = 0.$$
Notice that these give the Stein complementary series.
    \item
If $\nu+1=m$ is a non negative integer, then
$$
\ker D^-=
\left( \oplus_k \mathbb{C}w_k\otimes 1 \right) \oplus
\left( \oplus_{N>0} \mathbb{C} w_m\otimes Y^N \right)\,.
$$
The inequalities now become $(m+1)N\geq 0$, which are always satisfied.
Moreover, 
$$\ker D^- \cap \ker D^+ = \mathbb{C}w_{-m}\otimes 1$$ 
is the highest weight of the Discrete series whose $SO(2)$-structure is
$$\oplus_{k\leq -m}\mathbb{C}w_k\otimes 1,$$ 
so the space $\ker D^+\cap \ker D^-$ of strongly harmonic symplectic spinors recovers the lowest $K$-types of the discrete series.
\end{itemize}

\subsection{Hermitian symmetric spaces and Weyl algebra}
We now apply the previous discussion to the case where 
the pair $(\mg,\mh)$ is a Hermitian symmetric pair $(\mg,\mk)$.
More precisely, let $G/K$ be a Hermitian symmetric space of the non-compact type. In particular, $G$ is a non-compact connected semisimple Lie group with Lie algebra $\fg$ and $K$ is a maximal compact subgroup of $G$, associated with a Cartan involution $\theta$, with Lie algebra $\fk$. Moreover, denoting by $B$ the (non-degenerate) Killing form of $\fg^\mathbb{C}$, there is a $B$-orthogonal decomposition
\begin{equation}
\fg=\fk\oplus\fp\text{ with } 
\fp^\CC=\fp^+\oplus\fp^-
\end{equation}
such that 
\begin{enumerate}
\item the restriction of $B$ to $\fp$ is $K$-invariant and positive-definite;
\item $[\fp^+,\fp^+]=[\fp^-,\fp^-]=0$, $[\fp^+,\fp^-]\subseteq \fk$; 
\item $[\fk,\fk]\subseteq \fk,$ $[\fk,\fp^+]\subseteq \fp^+$,  $[\fk,\fp^-]\subseteq \fp^-$;
\item $\fp^+$ and $\fp^-$ are irreducible $\fk$-modules;
\item $B(\fk,\fp)=B(\fp^+,\fp^+)=B(\fp^-,\fp^-)=0$;
\item $\fk$ has a one-dimensional center $\mathfrak{z}(\fk)$;
\item there exists a Cartan subalgebra ${\mathfrak t}^\CC$ of $\fg^\CC$ contained in $\fk^\CC$.
\end{enumerate}
    In this case, the symmetric pair $(\fg,\fk)$ is said to be Hermitian. The classical and exceptional Hermitian symmetric pairs $(\fg,\fk)$ are (see for instance \cite[Table V, p. 578]{He}): $(\fs{\mathfrak u}(p,q),\fs({\mathfrak u}(p)\times {\mathfrak u}(q)))$, $1\leq p\leq q$, $(\fSp(n,\bR),{\mathfrak u}(n))$, $n\geq 1$, $(\fs\fo(2,n),\fs({\mathfrak o}(2)\times{\mathfrak o}(n)))$, $n\geq 1$, $(\fs\fo(2n),{\mathfrak u}(n))$, $n\geq 1$, $(\fe_{6(-78)},\fs\fo(10)+\bR)$ and $(\fe_{7(-133)},\fe_6+\bR)$. 

We consider the endomorphism $J$ of
$\fp^\CC$ acting scalarly on $\fp^\pm$ by $\pm i$. Then, $J$ is real,
$\mathrm{ad}(\fk)$-invariant and satisfies $J^2=-1$. Define the following symplectic form $\omega$ on $\fp^\mathbb{C}$:
\begin{equation}
\begin{aligned}
\omega(E,E')&=\omega(F,F')=0 \,,\\
\omega(E,F)&=-\omega(F,E)=iB(E,F),\text{ for all }E,E'\in\fp^+,\;F,F'\in\fp^-.
\end{aligned}
\end{equation}
Note that the factor $i$ is used in order to make $\omega$ real, with
$B_\mathfrak{p}=B_{|\fp\times\fp}$ where 
$B_\fp:=\omega(\cdot,J\cdot)$ is as in the previous section. 
Let $\cW(\fp^\CC)=\cW(\fp^\CC,\omega)$ be the Weyl algebra generated by $\fp$ and the form $\omega$. Let $\mathcal{U}(\fg^\CC)$ and $\mathcal{U}(\fk^\CC)$ denote the universal enveloping algebras of $\fg^\CC$ and $\fk^\CC$, respectively. 

As in the previous sections, let $(Z_i^+)$ be a basis of $\fp^+$, and $(Z_i^-)$ the $\omega$-dual basis of $\fp^-$. Since the adjoint action of $\fk$ on $\fp$ preserves $B$, it also preserves the symplectic form $\omega$. This gives a Lie algebra map $\ad \colon \fk\to \mathfrak{sp}(\fp,\omega)$,
and viewing $\mathfrak{sp}(\fp,\omega) \subset \cW(\fp^\CC)$
we obtain
$$
\tau: \fk\to \cW(\fp^\CC)
$$
and we define the diagonal embedding of $\fk$
\begin{equation}
\delta \colon \fk \to 
\mathcal{U}(\fg^\CC)\otimes \cW(\fp^\CC)\,,
\;\;H\mapsto H\otimes 1+1\otimes \tau(H)\,.
\end{equation}
Moreover, if $(w_j\mid 1\le j\le m)$ denotes an orthonormal basis of $\fk^\CC$, then $(w_j,Z_k^+,Z_k^-)$ is a basis of $\fg^\CC$ with $B$-dual basis $(w_j,iZ_k^-,iZ_k^+)$. Now, the \emph{Casimirs} of $\fg$ and $\fk$ are, respectively,
\begin{equation}
\Omega_\fg=\sum_j w_j^2 +i\sum_k(Z_k^+Z_k^-+Z_k^- Z_k^+)\in \mathcal{Z}(\fg^\CC)\;\text{ and }\;
\Omega_\fk=\sum_j w_j^2\in \mathcal{Z}(\fk^\CC).
\end{equation}

In order to apply Parthasarathy's formula in the next section, we will
need a precise knowledge of the action of
$\fk$ on $S(\fp^-)$. In fact, the adjoint action
of $\fk$ on $\fp^-$ that can be extended to 
$S^N(\mathrm{ad})$ on $S^N(\fp^-)$ on one hand, and the action of $\fk$ through 
the map $\varphi \circ \tau$ where $\varphi \colon \cW(\fp^\mathbb{C})\to \mathrm{End}(S(\fp^-))$ on the other hand, do not coincide, as we shall see, especially because of the non trivial action of the center of $\fk$.
The action $S^N(\mathrm{ad})$ was studied
by Schmid. We will then express
$\varphi \circ \tau$ in terms of $S^N(\mathrm{ad})$ for each $N\geq 0$.

Denote by $\Pi$ (resp. $\Pi_c$, $\Pi_n$) the set of ${\mathfrak t}^\CC$-weights in $\fg^\CC$ (resp. $\fk^\CC$, $\fp^\CC$). Fix positive systems $\Pi^+$, $\Pi_c^+$ and $\Pi_n^+$ such that 
\[
\Pi^+=\Pi_c^+\cup\Pi_n^+.
\]
and $\fp^\pm$ are generated by the weight vectors
in $\Pi^\pm_n$ respectively. Note that the representations $\fp^
\pm$ are the two $\fk$-invariant proper subspaces of $\fp^\mathbb{C}$. In particular, $\Pi^+_n$ is determined up to sign.

In particular, if $\rho(\fg)$, $\rho(\fk)$ and $\rho(\fp)$ denotes half the sums of positive roots in $\Pi^+$, $\Pi_c^+$ and $\Pi_n^+$ respectively, one has
\begin{equation}\label{rhos}
\rho(\fg)=\rho(\fk)+\rho(\fp).
\end{equation}
Let us write
\begin{equation}\label{p+-}
\fp^\pm=\sum_{\pm\alpha\in\Pi^+_n}\fg_\alpha^\mathbb{C}.
\end{equation}
For each positive root $\alpha\in\Pi^+$, pick an $\mathfrak{sl}_2$-triple $(E_\alpha,E_{-\alpha},H_\alpha)$ : $E_{\pm\alpha}$ is a root 
vector associated to $\pm\alpha$, and
$$
[E_\alpha,E_{-\alpha}]= H_\alpha\,,\qquad
[H_\alpha,E_\alpha]=2E_\alpha \,\qquad
[H_\alpha,E_{-\alpha}]=-2E_{-\alpha}\,.
$$
Let us enumerate the non compact positive roots as 
$(\alpha_1,\ldots,\alpha_n)$ with $2n=\mathrm{dim}_\RR(\fp)$ and choose
\begin{equation}\label{rootbasis}
Z_{k}^+=Z_{\alpha_k}^+=E_{\alpha_k},\;\;Z_k^-=Z_{\alpha_k}^-=-iE_{-\alpha_k}\,.
\end{equation}
In particular, the element 
\begin{equation}\label{zgenerator}
Z=\frac{1}{2}\sum_{\alpha \in \Pi_n^+} H_\alpha =\frac{i}{2}\sum_k\lbrack Z_k^+,Z_k^-\rbrack
\end{equation}
generates $\mathfrak{z}(\fk)$ and coincides with the dual root vector $\rho(\fp)^\vee$ to $\rho(\fp)$. This is because a root $\alpha\in\Pi$ is compact if, and only if, $\alpha$ vanishes on $Z$, see \cite[Lemma 7.127]{Kn}. Here it is meant that $\beta(\rho(\fp)^\vee)=\langle\beta,\rho(\fp)\rangle = B(\beta^\vee,\rho(\fp)^\vee)$ for any linear form $\beta$ on ${\mathfrak t}$. 
In particular, we have 
\begin{equation}
\label{rho_perp}
    \langle \rho(\fk),\rho(\fp)\rangle=0\,.
\end{equation}
The following lemma will be useful. 
\begin{lem}\label{lemalpha} For every $w\in\fk$, one has on $S^N(\fp^-)$ for $N\geq 0$,
$$
\varphi(\tau(w))=-B(w,Z)
\mathrm{Id}+S^N(\ad)(w)\,.
$$
In particular, one has
$$
\varphi(\tau(Z))=-2n-\mathrm{deg}\,.
$$
where $\mathrm{deg}$ is the degree operator on
$S(\fp^-)$. Moreover if $w\in[\fk,\fk]$, then $\varphi \circ \tau (w) = S^N(\mathrm{ad})w$.
\end{lem}
\begin{proof}
By Proposition \ref{unitary}, one has:
$$
\varphi(\tau(w))(Z_{k_1}^-\cdots Z_{k_N}^-)
=\frac{i}4\sum_{kl}
\omega([w,Z_k^+],Z_l^-)
\varphi(Z_l^+Z_k^-+Z_k^-Z_l^+)
(Z_{k_1}^-\cdots Z_{k_N}^-)\,.
$$
Moreover, we have
\begin{align*}
    \varphi(Z_l^+Z_k^-)(Z_{k_1}^-\cdots Z_{k_N}^-)
    &=2i\delta_{kl}Z_{k_1}^-\cdots Z_{k_N}^- \\
    &\phantom{2i\delta_{kl}}+2i\sum_{j}\delta_{lk_j}Z_k^-
    Z_{k_1}^-\cdots\widehat{Z_{k_j}^-}\cdots Z_{k_N}^- \\
    \varphi(Z_k^-Z_l^+)(Z_{k_1}^-\cdots Z_{k_N}^-)
    &=2i\sum_j \delta_{lk_j}Z_k^-
    Z_{k_1}^-\cdots\widehat{Z_{k_j}^-}\cdots Z_{k_N}^-\,,
\end{align*}
so that
\begin{align*}
    \varphi(\tau(w))(Z_{k_1}^-\cdots Z_{k_N}^-)
&= -\frac12\sum_k\omega([w,Z_k^+],Z_k^-)
(Z_{k_1}^-\cdots Z_{k_N}^-) \\
&\phantom{=-\frac12}-\sum_{kj} \omega([w,Z_k^+],Z_{k_j}^-)Z_k^-Z_{k_1}^-\cdots \widehat{Z_{k_j}^-}\cdots Z_{k_N}^- \\
&=-B(w,Z)Z_{k_1}^-\cdots Z_{k_N}^- \\
&\phantom{=-\frac12}
+\sum_{kj} \omega(Z_k^+,[w,Z_{k_j}^-])Z_k^-Z_{k_1}^-\cdots \widehat{Z_{k_j}^-}\cdots Z_{k_N}^- \\
&=-B(w,Z)Z_{k_1}^-\cdots Z_{k_N}^- +\sum_j Z_{k_1}^-\cdots 
[w,Z_{k_j}^-]\cdots Z_{k_N}^- \,.
\end{align*}
When $w=Z$, one has $\mathrm{ad}(Z)=\pm \mathrm{Id}$ on $\fp^\pm$. It follows that
$B(Z,Z) = \mathrm{Trace} (\mathrm{ad}(Z)^2)
=2n$ and $S^N(\mathrm{ad})(Z)=-N\;\mathrm{Id}$.
\end{proof}

\begin{corollary}
\label{shift}
Let $V$ be an irreducible submodule 
of $S(\fp^-)$ for the adjoint representation of $K$ with highest weight $\gamma$. Then it is an irreducible submodule of $S(\fp^-)$ for the representation $\varphi \circ \tau$ with highest weight $\gamma-\rho(\fp)$.
\end{corollary}
\begin{proof}
If $v$ is a highest weight vector of $V$,
then it is a generating vector of $V$ under the representation of $K$. It is then also
generating for the action of the semisimple part of $K$ because its center acts as a scalar. It is then a generating vector for the representation 
$\varphi \circ \tau$ since both representations
coincide on the semisimple part of $K$. It follows that $V$ is irreducible under the representation $\varphi \circ \tau$. Now, if $H \in \mathfrak{t}$, we have
$$
\varphi\circ\tau (H)v=-B(H,Z)v+S^N(\mathrm{ad})(H)(v)=
-\rho(\fp)(H)v+\gamma(H)v= (\gamma-\rho(\fp))(H)v \,.
$$
Hence $v$ is a highest weight vector of $V$ for the representation $\varphi \circ \tau$ with weight $\gamma-\rho(\fp)$.
\end{proof}

\subsection{Schmid modules}
\label{section-schmid}
According to Corollary~\ref{shift}, the representation $\varphi \circ \tau$
can be computed based on the adjoint representation of $K$ on $S^N(\fp^-)$. The latter was computed by W. Schmid \cite{schmid}. Let us recall briefly the facts we will need.
 Schmid describes $S({\mathfrak p}^+)$ in terms of strongly orthogonal roots. Note that, as a representation of $K$, ${\mathfrak p}^-$ is the contragredient representation of ${\mathfrak p}^+$, i.e.
 ${\mathfrak p}^-=({\mathfrak p}^+)^*$,
 so the $K$-types in $S({\mathfrak p}^-)$ in the adjoint action can be deduced from those of $S({\mathfrak p}^+)$ by duality.

\begin{itemize}
\item Let $\Pi^+$ be the set of positive roots. Let $\gamma_1$ be the highest root in $\Pi^+$. Then $\gamma_1$ must be non-compact. 
\item Remove from $\Pi^+$ all roots connected to $\gamma_1$. One gets a subset $\Pi^+_1$. 
\item Fix $\gamma_2\in\Pi^+_1$ the highest root and remove all roots in $\Pi^+_1$ connected to $\gamma_2$. One gets a subset $\Pi^+_2$.
\item Continue the process and get a set $\{\gamma_1,\gamma_2,\cdots,\gamma_r\}$. These $\gamma_i$'s are the strongly orthogonal roots, where $r$ is he real rank of $G$. 
\item Define 
\begin{equation}\label{e:Schmid}
s_i=\gamma_1+\gamma_2+\cdots+\gamma_i \text{ for } i=1,2,\cdots,r.
\end{equation}
These $s_i$'s are the highest weights of the basic Schmid modules.
\item The highest weights of the irreducibles in $S({\mathfrak p}^+)$ are $s=a_1s_1+\cdots+a_rs_r$ with $a_j\in{\mathbb N}$ and they all occur with multiplicity one.
\item It follows that the $K$-irreducibles in $S^N({\mathfrak p}^+)$, the submodule of homogeneous elements of degree $N$, have highest weights: $s=a_1s_1+\cdots+a_rs_r$ with $a_j\in{\mathbb N}$,
$\sum_j j a_j=N$.
\item If $s=\sum a_j s_j$ is the highest weight of an irreducible
$K$-module in $S^N(\fp^+)$, then for any $k\in \mathbb{N}^*$, the
weight $ks=\sum k a_j s_j$ is the highest weight
of an irreducible
$K$-module in $S^{kN}(\fp^+)$.
\item By duality, the {\it lowest} weights of the irreducible $K$-subrepresentations of $S^N(\fp^-)$ are $-s$, where $s=a_1s_1+\cdots+a_rs_r$ with $a_j\in{\mathbb N}$,
$\sum_j j a_j=N$, as before. Hence, the highest weights of the irreducible $K$-subrepresentations of $S^N(\fp^-)$ are $-w_0s$, where $w_0$ is the longest Weyl element in the Weyl group of $K$.
\end{itemize}
In the sequel an irreducible $K$-submodule $V_\gamma$
of $S^N(\fp^-)$ in the representation $\varphi \circ \tau$ will be called a \emph{twisted Schmid module}.

\subsection{Symplectic Dirac inequalities}
As it was explained before,
the Weyl algebra has an involution 
associated to the polarization $J$,
namely $c(v)^*=a(v)$ on the generators. Moreover the hermitian
form $h$ on $\fp$ leads to a pre Hilbert space structure on $S(\fp^-)$, say $\langle\,;\,\rangle_{S(\fp^-)}$, such that $\varphi$ is an involutive 
representation of the Weyl algebra.
This is a unitary representation of $K$ as well, whose differential is
$\varphi \circ \tau$ that we have just computed.

Next we extend, to an involution of $\mathcal{U}(\fg^\mathbb{C})$, the complex conjugation $X\mapsto \overline{X}$ of $\fg^\mathbb{C}$ with respect to its real form $\fg$, 
so that:
\[\overline{\fp^-}=\fp^{+}.\]
We also define an involution on $\mathcal{U}(\fg^\CC)$ (i.e., a conjugate-linear involutive map $*$
such that $(xy)^* = y^*x^*$)
to be the anti-automorphism extending the mapping 
\[
\fg\ni X\mapsto X^*:=-\overline{X}.
\]
Both involutions on $\mathcal{U}(\fg^\CC)$ and
$\mathcal{W}(\fp^\CC)$ lead to 
an involution on the tensor product
$\mathcal{U}(\fg^\CC)\otimes\mathcal{W}(\fp^\CC)$
and we deduce that 
\begin{equation}\label{starD}
(D^\pm)^*=D^\mp\;\text{ in }\;\mathcal{U}(\fg^\CC)\otimes\cW(\fp^\CC).
\end{equation}

Finally, suppose $(\pi,V)$ is a unitary $(\fg,K)$-module equipped with an invariant positive definite Hermitian form 
$\langle\;,\;\rangle_V$. There is a unitary structure on the tensor product $V\otimes S(\fp^-)$ given by:
\begin{equation}\label{unitTensor}
\langle\;,\;\rangle_{V\otimes S(\fp^-)}:=\langle\;,\;\rangle_V\;\langle\;,\;\rangle_{S(\fp^-)}.
\end{equation}
This unitary structure leads to a unitary 
representation of $K$. Any irreducible $K$-submodule $V_{\tilde{\mu}}$ of type $\tilde{\mu}$ in $V\otimes S(\fp^-)$ is contained in some $K$-invariant subspace $V_\mu \otimes V_\gamma$ for some 
$V_\mu \subset V$ an irreducible $K$-submodule of $(\pi,V)$ and a twisted Schmid module $V_\gamma$ in $S^N(\fp^-)$. For $N\geq 0$, write $S^N(\fp^-)$ for the subspace of $S(\fp^-)$ of homogeneous elements of degree $N$. Then, we have
$$
V\otimes S^{N-1}(\fp^-)
\mathop{\myrightleftarrows{\rule{7mm}{0cm}}}\limits^{D^+}_{D^-}
V\otimes S^{N}(\fp^-)
\mathop{
     \myrightleftarrows{\rule{7mm}{0cm}}
}\limits^{D^+}_{D^-}
V\otimes S^{N+1}(\fp^-)\,.
$$

Now, if $0\neq \tilde v\in V\otimes S(\fp^-)$ belongs to the kernel of $D^-$, one has 
$$[D^+,D^-]\tilde v=-D^-D^+\tilde v$$
and
\begin{align*}
\langle [D^+,D^-]\tilde v, \tilde v\rangle_{V\otimes S(\fp^-)}&=\langle -D^-D^+\tilde v,\tilde v\rangle_{V\otimes S(\fp^-)}\\
&=-\langle D^+\tilde v,D^+\tilde v\rangle_{V\otimes S(\fp^-)}\;\text{ by }(\ref{starD}).
\end{align*}
Thus, we obtains that:
\begin{equation}\label{Dirac-ineq}
\langle [D^+,D^-]\tilde v, \tilde v\rangle_{V\otimes S(\fp^-)}\leq 0.
\end{equation}
Suppose $(\pi,V)$ has an infinitesimal character. Using \eqref{partha}, we then find
\begin{align}\label{Dirac-ineq-3} 
\pi(\Omega_\fg)\langle \tilde v,\tilde v\rangle_{V\otimes S(\fp^-)}\leq  & \langle (2 \pi(\Omega_\fk)\otimes 1+ 1\otimes\tau(\Omega_\mathfrak{k})-
 \delta(\Omega_\mathfrak{k}))\tilde v,\tilde v\rangle_{V\otimes S(\fp^-)}
\end{align}
The difficulty in applying this inequality to find effective bounds for $\pi(\Omega_\fg)$ comes from the fact that the right hand side involves \emph{both} the $\fk$-structure of $V$ and the diagonal $\fk$-structure of $V\otimes S(\fp^-)$. Note that for the classical (orthogonal) Dirac inequality, only the diagonal $\fk$-structure matters.

From now on, an irreducible $K$-module of highest weight $\nu$ will be denoted by $V_\nu$. Since $\ker D^-$ can be decomposed into irreducible $K$-submodules,
we may assume that $\tilde{v}$ lies in a $K$-invariant irreducible 
subspace $V_{\tilde{\mu}}$ with highest weight $\tilde{\mu}$. 
In turn the space $V_{\tilde{\mu}}$ has to appear
in the tensor product $V_\mu \otimes V_\gamma$ of  irreducible $K$-components
$V_\mu \subset V$ and a twisted Schmid module $V_\gamma \subset S(\mathfrak{p}^-)$.

We then obtain
\begin{theorem} \label{inequality1}
Let $(\pi,V)$ be a unitary $(\fg,K)$-module with infinitesimal character. Let $N\geq 0$. 
Then, for any non-trivial irreducible $K$-submodule $V_{\tilde{\mu}}$ occuring in $\ker D^- \cap (V\otimes S^N(\fp^-))$ such that $V_{\tilde{\mu}} \subset V_\mu\otimes V_\gamma$, where $V_\mu$ is an irreducible $K$-submodule of $V$ and $V_\gamma\subset S^N(\mathfrak{p}^-)$ is a twisted Schmid module, we have:
\begin{equation}
    \label{general_inequality}
\pi(\Omega_\mg) \leq
\mu(\Omega_\mk)+(\mu(\Omega_\mk)+\gamma(\Omega_\mk)-\tilde{\mu}(\Omega_\mk))\,.
\end{equation}
\end{theorem}

\begin{remark}
Notice that $\mu(\Omega_\mk)+\gamma(\Omega_\mk)-\tilde{\mu}(\Omega_\mk)$ is the action of the operator $-\mathcal{D}^0$
on $V_{\tilde{\mu}}\subset V_\mu\otimes V_\gamma$. 
\end{remark}

The level $N=0$ corresponds to $V_\gamma=\mathbb{C}_{-\rho(\fp)}$, and we have as $K$-modules
$$V\otimes \mathbb{C}_{-\rho(\fp)}=\ker D^-_{|V\otimes S^0(\fp^-)}\subset \ker D^-\,.$$ 
Moreover,
$V_{\tilde{\mu}}=V_\mu\otimes \mathbb{C}_{-\rho(\fp)}$, so that  the inequality~\eqref{general_inequality}
reads as follows. 

\begin{corollary}
For any $K$-type $\mu$ of $V$, we have
\begin{equation}
    \label{sympineq00}
    \pi(\Omega_\fg)  \leq  \|\mu+\rho(\fg)\|^2-\|\rho(\fg)\|^2.
\end{equation}
\end{corollary}
\begin{proof}
Apply (\ref{general_inequality}) with $\tilde\mu=\mu {-\rho(\fp)}$. 
Recall that if $V_\nu$ is an irreducible representation of $K$ with highest weight $\nu$, we have
$$
\nu(\Omega_\fk)=\langle \nu,\nu+2\rho(\fk) \rangle{ }=\|\nu+\rho(\fk)\|^2-\|\rho(\fk)\|^2
$$
so that
\begin{align*}
    \mu(\Omega_\fk){-\rho(\fp)}(\Omega_\fk)-
    (\mu-\rho(\fp))(\Omega_\fk) &=\langle \mu,\mu+2\rho(\fk)\rangle + 
    \langle{-\rho(\fp)},{-\rho(\fp)} +2\rho(\fk)\rangle \\
    &\qquad - \langle \mu-\rho(\fp),\mu-\rho(\fp)+2\rho(\fk)\rangle \\
    &=2\langle\mu,\rho(\fp)\rangle \,.
\end{align*}
It follows that
$$
    \pi(\Omega_\fg)) \leq 
    \langle \mu,\mu+2\rho(\fk)+2\rho(\fp)\rangle  =
    \|\mu+\rho(\fg)\|^2-\|\rho(\fg)\|^2\,.
$$
\end{proof}
\begin{remark}
As we shall see in Section~\ref{comparison}, the inequality~\eqref{sympineq00} coincides with the standard Parthasarathy inequality.
\end{remark}


\subsection{Kraljevi\'c corners and applications to lowest weight representations}
It turns out that for certain weights $\tilde{\mu}$ occurring in $\ker(D^{-})$, the inequality~\eqref{general_inequality} can be made explicit. 
\begin{definition}
    \cite{kraljevic}
    A $K$-type $\mu$ of $\pi$ is said to be a \emph{$\Pi^+$-corner} of $\pi$ if $\mu-\alpha$ is not a $K$-type of $\pi$ for every non compact positive root $\alpha$. More generally, if $\widetilde{\Pi}^+$ any positive root system containing $\Pi_c^+$, we will say that $\mu$ is a corner with respect to $\widetilde{\Pi}^+$, or a $\widetilde{\Pi}^+$-corner, if $\mu-\alpha$ is not a $K$-type of $V$ for all non compact positive root $\alpha$ in $\widetilde{\Pi}^+$.
\end{definition}

\begin{proposition}\label{p:corner}
    If $\mu$ is a $\Pi^+$-corner of $\pi$  then $V_{\mu}\otimes S(\fp^-) \subset \ker D^-$.
\end{proposition}

\begin{proof}
Observe that $D^-(V_{\mu}\otimes S(\fp^-))\subset \sum_{\alpha\in \Pi^+_n} V_{\mu-\alpha}\otimes S(\fp^-)$. The claim now follows from the definition of the $\Pi^+$-corner.
\end{proof}

\begin{theorem}\label{t:corner}
    Assume $\mu$ is a $\Pi^+$-corner of a unitary $G$-representation $\pi$ with infinitesimal character. 
    Then, for any twisted Schmid module $V_\gamma\subset S^N(\fp^-)$ at any level $N\ge 0$, we have
    \begin{equation}
    \label{corner-inequality}
        \pi(\Omega_\fg)\leq
        \langle \mu,\mu-2\gamma+2\rho(\fk) \rangle
        = \|\mu+\rho(\fg)\|^2
        -\|\rho(\fg)\|^2
        -2\langle \mu,\gamma+\rho(\fp)\rangle\,.
    \end{equation}
\end{theorem}

\begin{proof}
Using Proposition \ref{p:corner}, we may take $\tilde{\mu}=\mu+\gamma$ in the  equality~\eqref{baseequality} whenever $\mu$ is a $\Pi^+$-corner. We deduce that
$$
\mu(\Omega_\fk)+\gamma(\Omega_\fk)-\tilde{\mu}(\Omega_\fk)=-2\langle \mu,\gamma\rangle,
$$
and the inequality follows from (\ref{general_inequality}).
\end{proof}

\begin{corollary}\label{c:corner} Assume $\mu$ is a $\Pi^+$-corner of an irreducible $G$-representation $\pi$. If there exists $1\le i\le r$ such that $\langle \mu,-w_0s_i\rangle>0$, where $s_i$ is as in (\ref{e:Schmid}),  then $\pi$ is not unitary.
\end{corollary}

\begin{proof}
The condition on the existence of an $s_i$ with this property is equivalent to the existence of a $V_{\gamma'}$ in $S(\fp^-)$ (untwisted irreducible $K$-representation) such that $\langle\mu,\gamma'\rangle>0$. Assume this is the case and suppose, by contradiction, that $\pi$ is unitary. If $V_{\gamma'}$ occurs in $S(\fp^-)$, then $V_{N\gamma'}$, occurs as well for all $N\in\mathbb N$. Applying the inequality in Theorem \ref{t:corner} with $\gamma=N\gamma'-\rho(\fp)$, it follows that  $\pi(\Omega_\fg)<-n$ for all natural numbers $n$, a contradiction.
\end{proof}

\begin{remark}
It appears in the proof that if $V$ is unitary, then the best inequality in~\eqref{corner-inequality} is reached for $\gamma=-\rho(\fp)$, that is $N=0$. So the inequality reduces to ~\eqref{sympineq00} which is, as we have already noticed, equivalent to the standard Parthasarathy inequality.
\end{remark}

We consider lowest weight module, defined as follows. Let
$V_\mu$ be an irreductible representation of $K$ with highest weight $\mu$.
Form the representation $\pi_{\mu}$ defined as the left $\mathcal{U}(\fg^\mathbb{C})$-module $\mathcal{U}(\fg^\mathbb{C}) \otimes_{\mathcal{U}(\mathfrak{k}^\mathbb{C} \oplus \mathfrak{p}^-)} V_{\mu}$. 
It follows from the Poincar\'e-Birkhoff-Witt theorem that the restriction of $\pi_\mu$ to $K$ is:
\begin{equation}
    \pi_{\mu}|_K = V_{\mu} \otimes_{\mathbb{C}} S(\mathfrak{p}^+)\,.
\end{equation}
The weight $\mu$ is a lowest weight
of $\pi_\mu$. In other words, $\mu$ is a $\Pi^+$-corner of the generalized Verma module $\pi_\mu$. Moreover, $\pi_\mu$ has a unique irreducible quotient $\overline{\pi}_\mu$, and $\mu$ is still a $\Pi^+$-corner of this quotient.

\begin{proposition}
If there exists a Schmid module $V_\gamma \in S(\fp^-)$, 
such that 
$\langle \mu,\gamma\rangle > 0$, then $\overline{\pi}_\mu$ is
not unitary.
\end{proposition}
\begin{proof}
Follows from Corollary \ref{c:corner}. As noted there, it is enough to check this condition when $\gamma=-w_0s_i$, for the basic Schmid highest weights $s_i$.
\end{proof}

\noindent {\bf Application to $SU(p,q)$.} $p\le q$, $p+q=n$. The highest weights of $K$-modules are of the form
\[
\mu=(\mu_1,\dots,\mu_p\mid \mu_{p+1},\dots,\mu_n),\quad\text{with } \mu_1\ge\dots\ge\mu_p,\ \mu_{p+1}\ge\dots\ge\mu_n,
\]
and $\mu_i-\mu_j\in \bZ_{\ge 0}$ if $i,j\in\{1,\dots,p\}$ or $i,j\in\{p+1,\dots,n\}$. The positive roots $\Delta^+_c$ are 
\[
\Delta_c^+ = \{ \epsilon_i - \epsilon_j \mid 1 \le i < j \le p \text{ or }  p+1 \le i < j \le n\} \,.
\]
The highest noncompact root is $\gamma_1=\epsilon_1-\epsilon_n$. The basic Schmid modules have highest weights
\[
s_i=(\underbrace{1,\dots,1}_i,0,\dots,0\mid 0,\dots,0,\underbrace{-1,\dots,-1}_i), \quad 1\le i\le p.
\]
The Weyl group of $K$ is $S_p\times S_q$. If $w_0$ is the longest Weyl group element of $S_p\times S_q$, we have
\[
-w_0s_i=(0,\dots,0,\underbrace{-1,\dots,-1}_i\mid \underbrace{1,\dots,1}_i, 0,\dots,0), \quad 1\le i\le p.
\]
Then $\langle\mu,-w_0s_i\rangle=(\mu_{p+1}+\dots+\mu_{p+i})-(\mu_{p-i+1}+\dots+\mu_p).$ 
If this is a positive number for some $i$, then $\bar\pi_\mu$ is not unitary. But notice that this is equivalent to $\langle\mu,-w_0s_1\rangle$ being positive. Indeed, if  $\langle\mu,-w_0s_1\rangle=\mu_{p+1}-\mu_p\le 0$, then  $\mu_{p+2}\le\mu_{p+1}\le \mu_p\le\mu_{p-1}$, so $\langle\mu,-w_0s_2\rangle\le 0$ etc. Hence we have the following

\begin{corollary}
For $SU(p,q)$, if $\mu_{p+1}>\mu_p$ then the lowest weight module $\bar\pi_\mu$ is not unitary.
\end{corollary}

The classification of unitary highest (or lowest) weight modules is well known; the point of this corollary is to illustrate how the symplectic Dirac inequality can be used to get explicit non-unitarity criteria.

\subsection{Higher levels and PRV components}
For higher levels $N$, the tensor product $V_\mu\otimes V_\gamma$ need not be irreducible.
However, it is clear that
\begin{equation}
\pi(\Omega_\mg) \leq
2\mu(\Omega_\mk)+
\gamma(\Omega_\mk)-
\mathop{\text{Min}}
\left\{\nu(\Omega_\mk)\,,\,V_\nu\subset V_\mu\otimes V_\gamma\right\}\,.
\end{equation}
On the other hand, it is known that the minimum 
of $\nu(\Omega_\fk)$ is reached at the PRV component. More precisely, for an element $h \in \mathfrak{t}^*$, let $h^+$ denote the unique $W(K)$-dominant conjugate of $h$ in the Weyl group $W(K)$ of $K$. Let $w_0$ be the longest element in $W(K)$. Then the PRV component of $V_\mu\otimes V_\gamma$ is $V_{(\mu+w_0\gamma)^+}$, while its Cartan component is 
$V_{\mu+\gamma}$.
Both occur with multiplicity one
\cite[Corollary 1 to Theorem 2.1]{PRV}.
\begin{lemma}(\cite[Lemma 3.2.6]{vogan},\;\cite[Lemma 9.1.6]{wal})
\label{PRVinequality}
We have for any $V_\nu \subset V_\mu \otimes V_\gamma$,
$$
\|(\mu+w_0\gamma)^+ +\rho(\fk)\|\leq \|\nu+\rho(\fk)\|\leq 
\|\mu+\gamma+\rho(\fk)\|
\,,
$$
\end{lemma}
After squaring and substracting $\|\rho(\fk)\|^2$,
we also get
$$
\langle \nu, \nu + 2\rho(\fk) \rangle\geq 
\langle (\mu+w_0\gamma)^+ , (\mu+w_0\gamma)^++2\rho(\fk) \rangle\,.
$$
Combining Cartan and PRV components, we deduce that
\begin{align}
\mu(\Omega_\fk) + &\gamma(\Omega_\fk) -\tilde{\mu}(\Omega_\fk) \notag \\
& = \langle \mu, \mu+2\rho(\mk) \rangle +
\langle \gamma, \gamma +2\rho(\fk) \rangle
- \langle \tilde{\mu},\tilde{\mu} + 2\rho(\fk)
\rangle
\label{baseequality}
 \\
& \leq \langle \mu, \mu+2\rho(\mk) \rangle + 
\langle \gamma, \gamma +2\rho(\fk) \rangle
- \langle (\mu+w_0\gamma)^+,(\mu+w_0\gamma)^+ + 2\rho(\fk) \rangle
\notag
\\
& =  \langle \mu+\gamma, \mu+\gamma+2\rho(\mk) \rangle - 
\langle (\mu+w_0\gamma)^+,(\mu+w_0\gamma)^+
+2\rho(\fk) \rangle
-2\langle \mu,\gamma\rangle
\notag
\\
&
=\|\mu+\gamma+\rho(\fk)\|^2-\|(\mu+w_0\gamma)^++\rho(\fk)\|^2 -2\langle \mu,\gamma \rangle
\,.
\end{align}
In particular, we have proved the following.
\begin{theorem}\label{t:ineq}
Assume $(\pi,V)$ is a unitary $(\fg,K)$-module with infinitesimal character.
Assume there exist irreducible submodules $V_\mu
\subset V$ and a twisted Schmid module $V_\gamma \subset S(\fp^-)$ such that $\ker D^- \cap V_\mu \otimes V_{\gamma}\neq 0$. Then for any such pair
$(\mu,\gamma)$ we have
\begin{equation}
\label{sympineq0}
\pi(\Omega_\mg) \leq \|\mu+\gamma+\rho(\fk)\|^2-\|(\mu+w_0\gamma)^++\rho(\fk)\|^2 +
\langle \mu,\mu -2\gamma + 2\rho(\fk) \rangle\,.
\end{equation}
\end{theorem}

\subsection{Level $N=1$.}
\label{level1}
In Section \ref{invpolker}, we have explained one procedure to construct elements of the kernel. In particular, we can realize nonzero vectors at level $N=1$ as follows.

Suppose that $0\neq v_0\in V_\mu\subset V$ for a $K$-type $\mu$ such that $\mu$ is {\it not} a $\Pi^+$-corner. We claim that $\pi(\mathfrak p^-)v_0\neq 0$. Otherwise, $\mathcal U(\mathfrak g)v_0=\mathcal U(\mathfrak p^+) V_\mu$ would be a submodule of $V$, therefore coinciding with $V$ by irreducibility. But all $K$-types in $V$ would have highest weights of the form $\mu+\mathbb Z_{\ge 0}\Pi_n^+$, making $v_0$ a $\Pi_n^+$-corner. Since $\pi(\mathfrak p^-)v_0\neq 0$, we may find $\varpi\in ({\bigwedge}^2 \mathfrak p^-)^*$ such that 
\[
v_\varpi=\sum_{i,j}\varpi (Z_i^-,Z_j^-) ~\pi(Z_i^-)v_0\otimes Z_j^-\neq 0\,.
\]
Then $0\neq v_\varpi \in \ker D^-\cap ~(V_{-w_0(s_1)}\otimes V_\mu)\otimes S^1(\mathfrak p^-)$, where $s_1$ still is the highest non-compact root. 

In this case, \eqref{sympineq0}, yields the following
\begin{corollary} As in Theorem \ref{t:ineq}, suppose $\ker D^-\cap V_\mu\otimes S^1(\fp^-)\neq 0$. Then, one has:
\begin{equation}
    \pi(\Omega_\mg)\leq \|\mu-w_0(s_1+\rho(\fg))\|^2-\|(\mu-s_1)^+-w_0\rho(\fg)\|^2 +
\langle \mu,\mu +2w_0(s_1) +2\rho(\fg) \rangle\,,
\end{equation}
where $s_1$ is the highest positive noncompact root.
\end{corollary}

\begin{proof}
This follows immediately from \eqref{sympineq0}, using $\gamma=-w_0(s_1)-\rho(\fp)$ and the fact that $\rho(\fp)$ is invariant under the Weyl group of $K$. Note that $w_0\rho(\fg)=-\rho(\fk)+\rho(\fp)$.
\end{proof}

The following proposition says when the Cartan component sits in the kernel of $D^-$.
\begin{proposition}
\label{cartan}
Let $\gamma_1=-w_0(s_1){-\rho(\fp)}$ be the highest weight
of $S^1(\fp^-)=\fp^-$ with respect to the representation
$\varphi \circ \tau$ and $v_{\gamma_1}=Z_{-w_0(s_1)}^-$ be the corresponding highest weight vector. Let $V_\mu$ be an irreducible $K$-submodule of the $(\fg^\mathbb{C},K)$-module $V$ and $v_\mu$ be a highest weight vector of $V_\mu$. Then the Cartan component $V_\mu \otimes S^1(\fp^-)$ is contained in the kernel of $D^-$ if and only if
$\pi(Z_{-w_0(s_1)}^-)v_\mu=0$.
\end{proposition}
\begin{proof}
Since $\ker D^-$ is $K$-invariant and the Cartan component of the considered representation is generated by the vector $v_\mu \otimes Z_{-w_0(s_1)}^-$, it is enough to prove that $D^-(v_\mu \otimes Z_{-w_0(s_1)}^-)=0$. Now,
\begin{align*}
    D^-(v_\mu \otimes Z_{-w_0(s_1)}^-)
    &=\sum_{\alpha\in\Pi^+_n}
        \pi(Z_\alpha^-)v_\mu\otimes \varphi(Z_\alpha^+)Z_{-w_0(s_1)}^- \\
    &=\pi(Z_{-w_0(s_1)}^-)v_\mu\otimes 1\,,
\end{align*}
and the proposition follows.
\end{proof}
Note that this condition is weaker than that of being a $\Pi^+$-corner. Assume from now on that the Cartan component $V_{\tilde{\mu}}$ of $V_\mu\otimes V_{\gamma_1}=V_\mu \otimes S^1(\fp^-)$ is in the kernel of $D^-$.
Then inequality~\eqref{general_inequality} reads
$$ 
\pi(\Omega_\fg)\leq \mu(\Omega_\fk)-2\langle \mu,\gamma_1 \rangle = \langle \mu,\mu+2\rho(\fk)-2\gamma_1 \rangle \,.
$$
It follows that this inequality is better than the standard Parthasarathy inequality if and only if 
$$
\langle \mu,\mu+2\rho(\fk)-2\gamma_1 \rangle 
< \langle \mu,\mu+2\rho(\fk){+2\rho(\fp)} \rangle 
$$
which in turn is equivalent to
$$
\langle \mu,\gamma_1 {+\rho(\fp)}  \rangle > 0 \,,
\text{ or equivalently }
\langle \mu,-w_0(s_1) \rangle > 0 \,.
$$
We have proved the following
\begin{proposition}
\label{cartanineq}
Assume that the Cartan component $V_{\tilde{\mu}}$ of $V_\mu\otimes V_{\gamma_1}=V_\mu \otimes S^1(\fp^-)$ lies in the kernel of $D^-$. Then inequality~\eqref{general_inequality} improves the standard
Parthasarathy formula if and only if
$$ \langle \mu,-w_0(s_1) \rangle > 0 \,.$$
\end{proposition}

\section{Symplectic versus Parthasarathy Dirac inequalities and cohomologies}
\label{comparison}

\subsection{Parthasarathy inequalities}

The Clifford algebra ${\bf Cl}(\fp^\mathbb{C})$ is the associative unital algebra defined as the quotient of the tensor algebra $T(\fp^\mathbb{C})$ by the ideal generated by the relations
\begin{equation*}
v\otimes w+w\otimes v-B(v,w) 1,\text{ for all }v,w\in \fp^\mathbb{C}.
\end{equation*}
${\bf Cl}(\fp^\mathbb{C})$ admits a unique, up to isomorphism, irreducible finite-dimensional representation, known as the spin module ${\bf S}$. In fact, ${\bf S}$ is isomorphic to the exterior algebra $\Lambda\fp^-$ of $\fp^-$, equipped with the Clifford multiplication $\zeta$. In particular, ${\bf S}$ is also a module for ${\mathfrak k}$ via the composition map
\begin{equation*}
\mathfrak{k}\stackrel{ad}\hookrightarrow\mathfrak{s}\mathfrak{o}(\mathfrak{\fp})\hookrightarrow{\bf Cl}(\fp)\stackrel{\zeta}\longrightarrow\text{End}({\bf S}).
\end{equation*}
As in symplectic case, a $\rho$-shift appears between the weights of ${\bf S}$ and those of $\wedge \fp^-$. 
Let $(x_i)$ be any bases of $\fp^\mathbb{C}$ and $(x^i)$ be the 
dual bases with respect to the non degenerate bilinear form $B_\fg$.
The \emph{algebraic Dirac operator} attached with any $(\fg^\mathbb{C},K)$-module $(V,\pi)$ is defined by
\begin{align}\label{algdirac}
D_V&:V\otimes{\bf S}\rightarrow V\otimes{\bf S}\\
D_V&=\sum_i\pi(x^i)\otimes\zeta(x_i) \,.
\end{align}
One can check that $D_V$ is $\fk$-equivariant and does not depend on the choice of the (orthonormal) basis for $\fp$. The Parthasarathy formula holds :
$$ D_V^2 = -(\pi\otimes 1(\Omega_\mg)+\|\rho(\mg)\|^2) + (\pi\otimes \zeta (\Omega_\mk) +\|\rho(\mk)\|^2) \,.$$ 
If, moreover, we assume that $(\pi,V)$ is a unitary $(\fg^\mathbb{C},K)$-module equipped with invariant non-degenerate Hermitian form $\langle\;,\;\rangle_V$, then there is an inner product on ${\bf S}$ for which $D_V$ is self-adjoint, and $D_V^2\geq 0$.
In particular, if $\tilde{\mu}$ is a $K$-type occurring in $V\otimes {\bf S}$, one gets the Parthasarathy-Dirac inequality (see \cite[Section 2]{P80}) as follows.
\begin{theorem}
Let $(V,\pi)$ be a unitary $(\fg^\mathbb{C},K)$-module. Then,
for any $K$-type $V_{\tilde{\mu}}$ occurring in $V\otimes {\bf S}$
\begin{equation}\label{parthadirac}
\pi(\Omega_\mg)\leq\|\tilde{\mu}+\rho(\mk)\|^2-\|\rho(\fg)\|^2 \,.
\end{equation}
Moreover, if $V_{\tilde{\mu}} \subset \ker D_V$, then equality holds.
\end{theorem}
It turns out that the Parthasarathy-Dirac inequality provides an efficient tool to detect non-unitary modules. In many applications it appears to be enough to check these inequalities in some 
particular cases. 
Let $W^1$ is the subset of the Weyl group of $\fg$
defined by
\[
W^1=\{w\in W(\fg,{\mathfrak t})\mid w\Pi^+\supset\Pi_c^+\}.
\]
The \emph{basic Parthasarathy inequalities of the first kind} are
$$
\pi(\Omega_\fg) \leq
\|\mu+w\rho(\fg)\|^2-\|\rho(\fg)\|^2\,,
$$
where $w\in W^1$
and  $V_\mu$ is a $K$-type in $V$.

In fact, the weights $w\rho(\fg)-\rho(\fk)$ are the highest weight in ${\bf S}$ and 
$\mu+w\rho(\fg)-\rho(\fk)$ is the Cartan component of $V_\mu\otimes V_{w\rho(\fg)-\rho(\fk)}$.
In particular, taking $w=1$, we get the \emph{standard} Parthasarathy inequality.
\begin{thm}\label{sympvspartha}
Let $(\fg,\fk)$ be a Hermitian symmetric pair and let $(\pi, V)$ be any unitary $(\fg,K)$-module. 
The level-$0$ symplectic Dirac inequality \eqref{sympineq00} coincides with the standard Parthasarathy inequality.
\end{thm}
One may also use the PRV component 
$V_{(\mu-\tilde{\rho}(\fp))^+}$
of $V_\mu\otimes V_{w\rho(\fg)-\rho(\fk)}$.
Here $\tilde{\rho}(\fp)$ is the half sum
of non compact roots of some positive root system
$\widetilde{\Pi}^+$ containing $\Pi^+_c$. We get {\it basic Parthasarathy inequality of the second kind} :
\begin{equation}
\label{prv-partha}
   \pi(\Omega_\fg) \leq \|(\mu-\tilde{\rho}(\fp))^++\rho(\fk)\|^2-\|\rho(\fg)\|^2 \,.
\end{equation}

\subsection{Remarks on the kernel of $D^+$}

Let $(V,\pi)$ be a unitary represenation of $G$
with $(G,K)$ a Hermitian symmetric pair.
We consider $\ker D^+$ instead of $\ker D^-$,
unlike we have done so far expect in some examples.
\begin{example}
If $\pi$ is a discrete series representation
with lowest $K$-type lying in 
$\widetilde{\Pi}^+=\Pi^+_c \cup (-\Pi^+_n)$ then
$\ker D^+$ contains this lowest $K$-type at level $0$.
\end{example}

The inequality (\ref{general_inequality})
becomes
$$
\pi(\Omega_\mg) \geq
\mu(\Omega_\mk)+
(\mu(\Omega_\mk)+\gamma(\Omega_\mk)-\tilde{\mu}(\Omega_\mk))\,,
$$
where now
\begin{align}
\mu(\Omega_\fk) + &\gamma(\Omega_\fk) -\tilde{\mu}(\Omega_\fk) \notag \\
& = \langle \mu, \mu+2\rho(\mk) \rangle +
\langle \gamma, \gamma +2\rho(\fk) \rangle
- \langle \tilde{\mu},\tilde{\mu} + 2\rho(\fk)
\rangle
\notag
 \\
& =  \langle \mu+\gamma, \mu+\gamma+2\rho(\mk) \rangle - 
\langle \tilde{\mu},\tilde{\mu}
+2\rho(\fk) \rangle
-2\langle \mu,\gamma\rangle
\notag
\\
&
=\|\mu+\gamma+\rho(\fk)\|^2-\|\tilde{\mu}+\rho(\fk)\|^2 -2\langle \mu,\gamma \rangle \notag\\
& \geq -2\langle \mu,\gamma \rangle \,.
\label{quite-good2}
\end{align}

Note that the equality in \eqref{general_inequality} holds on $\ker D^- \cap \ker D^+$. This space may be thought of as the space of \emph{strongly harmonic symplectic spinors}. 
Since $(D^-)^*=D^+$ while $\pi$ is unitary, 
one gets $(\mathrm{Im} D^\pm)^\perp = \ker D^\mp$ and 
$$
\ker D^-/\mathrm{Im} D^-\cap \ker D^-\simeq 
\ker D^- \cap \ker D^+ \simeq 
\ker D^+/\mathrm{Im} D^+\cap \ker D^+\,.
$$
Here, isomorphisms hold if  the representation $\pi$ 
is unitary, but this is not true in general. The quotient spaces are called \emph{symplectic Dolbeault cohomologies}
and can be defined for any $(\mathfrak{g}^\mathbb{C},K)$-module.
The following result suggests that these symplectic cohomology spaces might be of interest.
\begin{proposition}
Let $\pi$ be an irreducible unitary representation of $G$ with non zero symplectic Dolbeault cohomology, containing a $K$-type $V_{\tilde{\mu}}\subset V_\mu\otimes V_\gamma$.   
Then, one has:
$$ 
\pi(\Omega_\fg) = 2\mu(\Omega_\fk)+\gamma(\Omega_\fk) 
-\tilde{\mu}(\Omega_\fk)\,.
$$
\end{proposition}
\begin{proof}
It follows immediately from the symplectic Dirac inequality corresponding to the triple $(\mu,\gamma,\tilde{\mu})$.
\end{proof}

\subsection{The example of $SU(1,n)$}
Let $G=SU(1,n)$ with $n\geq 2$. We fix the following sets of positive roots:
\begin{align*}
\Pi^+&=\{\epsilon_i-\epsilon_j\mid 1\leq i<j\leq n+1\}\\
\Pi^+_c&=\{\beta_{ij}=\epsilon_i-\epsilon_j\mid 2\leq i<j\leq n+1\}\\
\Pi^+_n&=\{\alpha_j=\epsilon_1-\epsilon_j\mid 2\leq j\leq n+1\}\\
\end{align*}
with 
\begin{align*}
\rho(\fg)&=(\frac{n}{2},\frac{n-2}{2},\ldots,-\frac{n-2}{2},-\frac{n}{2})\\
\rho(\fk)&=(0,\frac{n-1}{2},\frac{n-3}{2},\ldots,-\frac{n-3}{2},-\frac{n-1}{2})\\
\rho(\fp)&=(\frac{n}{2},-\frac{1}{2},\ldots,-\frac{1}{2}).\\
\end{align*}
The positive roots system containing $\Pi_c^+$ are,
for $1\leq j\leq n+1$,
$$
\Pi^+_j=\Pi_c^+ \cup \Pi_{n,j}^+
\text{ with }
\Pi_{n,j}^+=\{-\alpha_2,\ldots,-\alpha_j\}
\cup
\{\alpha_{j+1},\ldots,\alpha_{n+1}\}
$$
(note that $\Pi^+_{n,1}=\Pi_n^+$ and 
$\Pi^+_{n,n+1}=-\Pi_n^+$) and corresponding half sums
$$\rho_j(\fg)=\rho(\fk)+\rho_j(\fp)
\text{ with }
\rho_j(\fp)=\Big(\frac{n}{2}-j+1,\underbrace{\frac12,\ldots,
\frac12}_{j-1},\underbrace{-\frac12,\ldots -\frac12}_{n-j+1}\Big)\,.$$
Fix $P=MAN$ a minimal parabolic subgroup in $G$, $(\tau, E_\tau)$ an irreducible representation of $M$, $\nu\in({\mathfrak a}^\CC)^*$, where ${\mathfrak a}$ is the Lie algebra of $A$, and consider the principal series representation of $G$ (normalized induction):
\[
I(\tau,\nu):=\text{ind}_P^G\tau\otimes e^{\nu+\rho_{\mathfrak a}}\otimes 1.
\]
Here $\rho_{\mathfrak a}$ is half the sum of positive ${\mathfrak a}$-roots in ${\mathfrak g}$. 
We fix $V=I(\tau,\nu)$. 
Following \cite[Section 2]{Se}, the irreducible representations of $M$ are parametrized by the set 
\begin{align*}
\Psi^+&=
\Bigl\{
x=(x_1,\ldots,x_{n+1})\in{\mathbb R}^{n+1}\mid\;x_2\in\frac{1}{n+1}{\mathbb Z},\;x_i-x_j\in{\mathbb Z}\;\text{ for }\;2\leq i<j\leq n,\\
&
\phantom{
\Bigl\{
x=(x_0,\ldots,x_n)\in{\mathbb R}^{n+1}\mid x_1 \Bigr\}
}
\text{ with }x_1=x_{n+1}=-\frac{1}{2}\sum_{2}^{n}x_j,\; \text{and }x_2\geq x_3\geq\cdots\geq x_{n}
\Bigr\}.
\end{align*}
Taking $\tau=\tau_x$, i.e., $\tau_x$ has highest weight $x$, for some $x\in\Psi^+$, the $K$-types of $V$ are given by (\cite[Theorem 1]{Se}):
\[
I(\tau_x,\nu)_{\mid_{K}}\simeq\bigoplus_{\mu\in\Phi_x}V_\mu
\]
where 
\begin{align*}
\Phi_x&=
\Bigl\{
\mu=(\mu_1,\mu_2,\ldots,\mu_{n+1})\mid\sum_{i=1}^{n+1} \mu_i=0,\;\mu_i-x_2\in{\mathbb Z}\text{ for }1\leq i\leq n+1,
\\& 
\phantom{=
\Bigl\{
\mu=(\mu_1,\mu_2,\ldots,\mu_{n+1})\mid\sum}
\text{and } \mu_2\geq x_2\geq\mu_3\geq x_3\geq\cdots\geq\mu_{n}\geq x_{n}\geq\mu_{n+1} \Bigr\}.
\end{align*}
Let
$
x=\left(x_1,x_2,\ldots,x_{n+1}\right)\in\Psi^+
$
and 
$
\mu'=\left(\mu_1',\mu_2',\ldots,\mu_{n+1}'\right)\in\Phi_x\,.
$
Note that $V$ has no $\Pi^+=\Pi_1^+$ corner because
$\mu'-\alpha_2 \in \Phi_x$ for any $\mu'\in\Phi_x$.
Moreover
$\mu'$ is a $\Pi^+_j$-corner of $V$
if and only if
$$
\mu_2'= x_2\geq\mu_3'= x_3\geq\mu_4'\cdots
x_{j-1}\geq\mu_j'=x_j=\mu_{j+1}'\geq x_{j+1}=\mu_{j+2}'
\geq \ldots x_{n-1}=\mu_{n}'\geq  x_{n}=\mu_{n+1}'\,.
$$
From now on we will assume that $\mu'$ is a $\Pi^+_2$-corner.

The highest weight of $\mathfrak{p}^+$ 
with respect to the adjoint action of $K$
is the largest non compact root 
$s_1=\alpha_{n+1}=\epsilon_1-\epsilon_{n+1}=(1,0,\ldots,0,-1)$. Note that $-w_0s_1=(-1,1,0,\cdots,0)=-\alpha_2$ here, and the corresponding twisted Schmid module $S^1(\fp^-)$ has highest weight $-w_0s_1-\rho(\fp)=-\alpha_2-\rho(\fp)$.

Let us now write explicitly the corresponding symplectic Dirac inequalities at level $1$. We know that $\ker D^-$ is nonzero at level $1$; let $v_\varpi\in\ker D^-$ be the level-1 vector defined in section \ref{level1}.
We have
$$
0\neq v_\varpi \in (\fp^-\otimes V_{\mu'})\otimes S^1(\fp^-) \,.
$$
Since all weights of $(\fp^-,\text{ad})$ have multiplicity $1$, 
$$\fp^-\otimes V_{\mu'}=
\oplus_{\alpha\in\Pi_n^+} m_\alpha V_{\mu'-\alpha}
\text{ with } m_\alpha\in\{0,1\}\,,$$
but $V_{\mu'-\alpha_j}$ is not a $K$-type in $V$ for 
$\alpha_j$ for $j\geq 3$,
as $\mu'$ is a $\Pi^+_2$-corner. We then get
$$0\neq v_\varpi \in \ker D^-\cap (V_{\mu'-\alpha_2}\otimes S^1(\fp^-))\,.$$
Let us now write $\mu=\mu'-\alpha_2$.
So if $\mu=(\mu_1,\ldots,\mu_{n+1})$, we 
have $\mu_1=\mu_1'-1$, $\mu_2=\mu_2'+1$
and $\mu_j=\mu_j'$ for $j\geq 3$.
Again by a multiplicity one argument, we have a decomposition
\begin{equation}\label{mjtens}
V_{\mu}\otimes S^1(\fp^-)=
\oplus_{\alpha_j\in\Pi_n^+}m'_j V_{\mu-\alpha_j-\rho(\fp)}
\text{ with } m'_j\in\{0,1\}\,.
\end{equation}
Note that the components of $v_\varpi$ in this decomposition are in $\ker D^-$ as well, because $D^-$
is $K$-equivariant. A least one of these components does not vanish. 

Assume that the Cartan component $V_{\mu-\alpha_2-\rho(\fp)}$ does not vanish. Then, one quickly checks that 
$$
\langle \mu,-w_0s_1\rangle 
=-\langle \mu,\alpha_2\rangle = \mu_2-\mu_1 \,,
$$
and proposition~\ref{cartanineq} implies that the symplectic Dirac inequality at level $N=1$ improves the standard
Parthasarathy formula if and only if \begin{equation}
\label{sun1}
    \mu_2-\mu_1 > 0\,.
\end{equation} 
Moreover, in
the right hand side of inequality~\eqref{general_inequality} we have
$\gamma=-\alpha_2-\rho(\fp)$ and $\tilde{\mu}=\mu-\alpha_2-\rho(\fp)$
\begin{align*}
    2\mu(\Omega_\fk)+
    &\gamma(\Omega_\fk)-\tilde{\mu}(\Omega_\fk) \\
    &=2\langle \mu,\mu+2\rho(\fk)\rangle \\
    &\phantom{2\langle\mu}+
    \langle -\alpha_2-\rho(\fp),
    -\alpha_2-\rho(\fp)+2\rho(\fk) \rangle \\
    &\phantom{2\langle\mu}-
    \langle \mu-\alpha_2-\rho(\fp),
    \mu-\alpha_2-\rho(\fp)+2\rho(\fk) \rangle \\
    &=\langle \mu,\mu+2\rho(\fk)\rangle +
    \langle \mu,2\rho(\fp)\rangle 
    +2\langle\mu,\alpha_2 \rangle\\
    &=\|\mu+\rho(\fg)\|^2-
    \|\rho(\fg)\|^2 
    +2\langle\mu,\alpha_2 \rangle\,.
\end{align*}
We then obtain
\begin{equation}
    \label{cartan_level1}
    \pi(\Omega_\fg)\leq \|\mu+\rho(\fg)\|^2-
    \|\rho(\fg)\|^2 +2\langle\mu,\alpha_2 \rangle \,.
\end{equation}
This leads to the same conclusion as \eqref{sun1}.
We may also compare this inequality with others basic Parthasarathy inequalities of the first kind. Our inequality improves all of them if and only if,
for all $i\geq 1$ we have
$$
\|\mu+\rho_i(\fg)\|^2-
    \|\rho(\fg)\|^2 > 
\|\mu+\rho(\fg)\|^2-
    \|\rho(\fg)\|^2 +2\langle\mu,\alpha_2 \rangle \,.
$$
But, $\rho_i(\fg)$ is in the Weyl group orbit of $\rho(\fg)$, so the inequality is equivalent to
$$
\langle \mu, \rho_i(\fp)-\rho(\fp)-\alpha_2\rangle >0 \,;
$$

Next step is to check the other components.
The same computation holds with $\alpha_2$ replaced with $\alpha_j$, $j\geq 3$ in the definition of $\tilde{\mu}$.
The right hand side of equation~\eqref{general_inequality} now reads
\begin{align*}
    2\mu(\Omega_\fk)+
    &\gamma(\Omega_\fk)-\tilde{\mu}(\Omega_\fk) \\
    &=2\langle \mu,\mu+2\rho(\fk)\rangle \\
    &\phantom{2\langle\mu}+
    \langle -\alpha_2-\rho(\fp),
    -\alpha_2-\rho(\fp)+2\rho(\fk) \rangle \\
    &\phantom{2\langle\mu}-
    \langle \mu-\alpha_j-\rho(\fp),
    \mu-\alpha_j-\rho(\fp)+2\rho(\fk) \rangle \\
    &=2\langle \mu,\mu+2\rho(\fk)\rangle \\
    &\phantom{2\langle\mu}+
    \langle -\alpha_2-\rho(\fp),
    -\alpha_2-\rho(\fp)+2\rho(\fk) \rangle \\
    &\phantom{2\langle\mu}-
    \langle \mu-\alpha_2-\rho(\fp)+\alpha_2-\alpha_j,
    \mu-\alpha_2-\rho(\fp)+\alpha_2-\alpha_j+2\rho(\fk) \rangle \\
    &=\|\mu+\rho(\fg)\|^2-
    \|\rho(\fg)\|^2 
    +2\langle\mu,\alpha_2 \rangle 
    +2\langle \mu-\rho(\fp)+\rho(\fk),\alpha_j-\alpha_2\rangle \,.
\end{align*}
because all roots have the same length.
But $\alpha_j-\alpha_2=\beta_{2j}$ is a sum of simple compact roots for $\Pi^+_c$,
and $\mu-\alpha_j-\rho(\fp)+\rho(\fk)=\tilde{\mu}+\rho(\fk)$ is dominant for $K$.
It follows that the last scalar product in the right hand side is positive.
So the inequality~\eqref{general_inequality} in this case is weaker than the previous inequality
\eqref{cartan_level1}, and reads 
\begin{equation}
    \label{parthasympgeneral}
    \pi(\Omega_\fg) \leq 
    \|\mu+\rho(\fg)\|^2-
    \|\rho(\fg)\|^2 
    +2\langle\mu,\alpha_2 \rangle 
    +2\langle \mu-\rho(\fp)+\rho(\fk),\alpha_j-\alpha_2\rangle \,.
\end{equation}

Moreover at least one of the following is true : if the symplectic inequality improves Parthasarathy inequality of the first kind, then for all $i$ we have,
\begin{equation}
    \|\mu+\rho_i(\fg)\|^2-
    \|\rho(\fg)\|^2 >
    \|\mu+\rho(\fg)\|^2-
    \|\rho(\fg)\|^2 
    +2\langle\mu,\alpha_j \rangle 
    +2\langle -\rho(\fp)+\rho(\fk),\alpha_j-\alpha_2\rangle \,.
\end{equation}
This amounts to
\begin{equation}
    \langle 
    \mu ,
    \rho_i(\fp)-\rho(\fp)-\alpha_j
    \rangle 
    +
    \langle 
    \rho(\fp)-\rho(\fk),
    \alpha_j-\alpha_2
    \rangle >0 \,.
\end{equation}

Let's us summarize this discussion.
\begin{proposition}
Let $\pi$ be a unitary irreducible representation 
of $SU(1,n)$. Assume that $\pi$ has a $\Pi^+$-corner $\mu'$. Then there exists $j$ such that $V_{\mu-\alpha_j-\rho(\fp)}  \subset (V_\mu\otimes S^1(\fp^-))\cap \ker D^-$ where $\mu=\mu'-\alpha_2$ is a $K$-type of $\pi$. Fix such a $j$ and write
$\mu=(\mu_1,\ldots,\mu_{n+1})$. Then the symplectic
Dirac inequality at level $1$ improves all basic Parthasarathy inequality of the first kind if and only if, for all $1\leq i \leq n+1$ and $2\leq j \leq n+1$ such that $m_j'\neq 0$ in \eqref{mjtens} and the component of $v_\varpi$ in
$V_{\mu-\alpha_j-\rho(\fp)}$ does not vanish, one has
\begin{equation}\label{final}
    (\mu_j-\mu_1)+\sum_{k=2}^i (\mu_k-\mu_1)
    \geq
    j-2
    \,.
\end{equation}
By convention, the sum is empty when $i=1$.
\end{proposition}

\begin{example}
Let $G=SU(1,3)$, $x=(-1,1,1,-1)$, $\mu'=(-3,1,1,1)$ and $\mu=(-4,2,1,1)$. Then the condition \eqref{final} is satisfied. In this case, the tensor product decomposition gives $\mu-\alpha_2-\rho({\mathfrak p})=(-13/2,7/2,3/2,3/2)$ (Cartan component), $\mu-\alpha_3-\rho({\mathfrak p})=(-13/2,5/2,5/2,3/2)$ (PRV component) and $\mu-\alpha_4-\rho({\mathfrak p})=(-13/2,5/2,3/2,5/2)$. The latter does not occur, so $m_4^\prime=0$ but both $m_2^\prime=m_3^\prime=1$. From \cite[Proposition 1]{kraljevic} the first reducibility point of the principal series $(\pi_{x,\nu})$ is given by $\nu=1$, and a sufficient condition for $\pi_{x,\nu}$ to be unitarizable is then $\nu<1$. The inequality \eqref{parthasympgeneral} gives 
$$ \frac{n^2}{2}+\frac{\nu^2}{2} \leq 1 -2\mu_j+2j \,,$$
that is $\nu^2\leq -7$ for $j=2$ and $\nu^2\leq 1$ for $j=3$. It follows that the Cartan component of the vector
$v_\varpi$ vanishes, and that  $\nu<1$ is a necessary condition for $\pi_{x,\nu}$ to be unitarizable. Note that the basic Parthasarathy of the second kind also leads to the same inequality as shown (in more generality) by Baldoni Silva and Barbasch \cite{baldoni-barbasch}. 
\end{example}

\bibliographystyle{plain}
\bibliography{biblio}

\end{document}